\documentclass[12pt]{article}

\usepackage{amsthm,amssymb,mathtools} 
\usepackage{color}
\usepackage[inline]{enumitem}
\usepackage{hyperref}
\usepackage{nicematrix}
\usepackage[subrefformat=parens]{subcaption}
\usepackage{subfiles}
\usepackage{multirow}
\usepackage[scr=rsfs]{mathalpha}
\usepackage{empheq}

\usepackage{algorithmic}
\usepackage{algorithm}

\usepackage[nameinlink,capitalise]{cleveref}

\usepackage{tikz-cd} 
\usepackage{tikz-3dplot}
\usetikzlibrary{patterns}
\usetikzlibrary{angles}

\usepackage{macros}
\usepackage{mathoperators}
\usepackage{theorems}

\hypersetup{
     colorlinks=true,
     linkcolor=black,
     filecolor=blue,
     citecolor = black,
     urlcolor=cyan,
     }

\def\@themcountersep{}
\usepackage{fancybox,color}
\definecolor{lred}{rgb}{1,0.8,0.5}
\definecolor{lblue}{rgb}{0.8,0.8,1}
\definecolor{dred}{rgb}{0.6,0,0}
\definecolor{dblue}{rgb}{0,0,0.7}
\definecolor{violet}{rgb}{0.5804,0.0000,0.8275}
\definecolor{purple}{rgb}{0.2400,0.5700,0.2500}
\definecolor{TGreen}{rgb}{0,0.50,0.10}

\usepackage[mathlines]{lineno} 
\usepackage{amsmath}

\usepackage{etoolbox} 

\usepackage{easy-todo}

\newcommand*\linenomathpatch[1]{%
    \cspreto{#1}{\linenomath}%
    \cspreto{#1*}{\linenomath}%
    \csappto{end#1}{\endlinenomath}%
    \csappto{end#1*}{\endlinenomath}%
}
\newcommand*\linenomathpatchAMS[1]{%
    \cspreto{#1}{\linenomathAMS}%
    \cspreto{#1*}{\linenomathAMS}%
    \csappto{end#1}{\endlinenomath}%
    \csappto{end#1*}{\endlinenomath}%
}

\expandafter\ifx\linenomath\linenomathWithnumbers
\let\linenomathAMS\linenomathWithnumbers
\patchcmd\linenomathAMS{\advance\postdisplaypenalty\linenopenalty}{}{}{}
\else
\let\linenomathAMS\linenomathNonumbers
\fi

\linenomathpatch{equation}
\linenomathpatchAMS{gather}
\linenomathpatchAMS{multline}
\linenomathpatchAMS{align}
\linenomathpatchAMS{alignat}
\linenomathpatchAMS{flalign}

\usepackage[numbers, sort]{natbib}

\title{Tight Conic Relaxations for Rank-one Doubly Nonnegative Matrix Completion}

\makeatletter
\let\@fnsymbol\@arabic
\makeatother

\author{
\normalsize
    Godai Azuma\thanks{Department of  Mathematical and Computing Science,
    Institute of Science Tokyo, 2-12-1-W8-29 Oh-Okayama, Meguro-ku, Tokyo 152-8550, Japan
    ({\tt azuma@comp.isct.ac.jp}, {\tt Makoto.Yamashita@comp.isct.ac.jp}).
    The research of Godai Azuma was supported by JSPS KAKENHI Grant Number JP24K20738.
    The research of Makoto Yamashita was partially supported by JSPS KAKENHI Grant Number 24K14836.}\;\textsuperscript{,}\thinspace
    \thanks{Department of Industrial and Systems Engineering,
    Aoyama Gakuin University, 5-10-1 Fuchinobe, Chuo-ku, Sagamihara-shi, Kanagawa 252-5258, Japan.}
\and
\normalsize
    Sunyoung Kim\thanks{Department of Mathematics, Ewha W. University, 52 Ewhayeodae-gil, Sudaemoon-gu,
    Seoul 03760, Korea  ({\tt skim@ewha.ac.kr}).
    }
\and
\normalsize
    Makoto Yamashita\footnotemark[1]
}

\begin{document}
\maketitle

\begin{abstract}
\noindent
We study tight conic relaxations for a quadratically constrained quadratic programming (QCQP) 
formulation of rank-one doubly nonnegative (DNN) matrix completion. 
Motivated by sparse QCQPs whose lifted matrix variables include elements not directly specified by the objective or constraints,
we interpret tightness as a rank-one completion property for the unspecified elements.
For sparsity patterns whose blocks consist of cycles and edges, we 
prove that the dual formulations associated with the DNN and completely positive (CP) relaxations 
are equivalent. For cycle-type sparsity patterns, we derive explicit sufficient conditions under which 
the semidefinite programming (SDP) and DNN relaxations are tight. These sufficient conditions are
stated explicitly in terms of
local ratio bounds and cumulative-difference conditions on a rank-one 
optimal solution. We also show that adding suitable edges to the sparsity pattern relaxes the ratio 
conditions required for tightness. The results provide tractable certificates for when conic relaxations 
recover a rank-one optimal solution of the underlying QCQP.
\end{abstract}

\vspace{0.5cm}


\vspace{0.5cm}

 



\section{Introduction} \label{sec:introduction}
Quadratically constrained quadratic programs (QCQPs) are fundamental nonconvex optimization problems in global optimization. 
Since QCQPs are generally hard to solve globally, convex conic relaxations such as semidefinite programming (SDP) 
and doubly nonnegative (DNN) programming are widely used to obtain tractable lower bounds and, in some cases, 
recover global solutions. We say that a relaxation is \textit{tight} (or \textit{exact}) if its optimal value coincides with that of the original QCQP.
If, in addition, the relaxation admits a rank-one optimal solution,
then an optimal solution of the original QCQP can also be recovered~\cite{kim2003exact,Sojoudi2014exactness}.

When the data matrices are sparse~\cite{Azuma2021,Azuma2022,Burer2019}, some elements of the lifted matrix 
variable $Z$ are not fixed directly by the objective function and constraints~\cite{fukuda2001exploiting,grone1984positive}. 
From this viewpoint, tightness can be interpreted as a matrix completion question: whether the unspecified elements can 
be chosen so that the resulting matrix remains feasible and has rank one. This observation connects the study of tight 
conic relaxations with structured rank-one matrix completion.

In this paper, we focus on rank-one doubly nonnegative matrix completion. More precisely, 
we consider matrix completion problems in which the completed matrix is required to be rank one, and elementwise nonnegative.
Let $\Omega \subseteq \{1,\ldots,m\} \times \{1,\ldots,n\}$ denotes the index set of the specified elements,
    and suppose that the elements $A_{ij}$ are specified for $(i, j) \in \Omega$.
A  low-rank matrix completion~\cite{Candes2009,WrightMa2022} can be described as  
\begin{equation} \label{eq:matrix_completion}
    \mathrm{find} \quad X \in \Real^{m \times n} \quad
    \subto        \quad X_{ij} = A_{ij}, \; (i, j) \in \Omega, \;\; \rank X \leq R.
\end{equation}
The rank-one ($R=1$) completion problem admits a natural QCQP formulation, 
and the associated SDP, DNN, and completely positive (CP) relaxations can therefore be analyzed 
within the framework of tight conic relaxations for nonconvex optimization problems.

The sparsity pattern of the specified elements plays a central role in this analysis. We represent it by the bipartite graph,
where each edge corresponds to a specified matrix element. Under this representation, the question of 
tightness becomes a graph-structured rank-one completion problem. Our goal is to identify graph-based 
and solution-based conditions under which the conic relaxations recover a rank-one optimal solution.

This work is related to two strands of literature. 
The first strand studies rank-one matrix completion problems
    through polynomial optimization formulations and analyzes the tightness of their convex relaxations~\cite{Cosse2021,Azuma2025}.
In these works, doubly nonnegative structure is not imposed on the completion matrix; instead, exact recovery is investigated 
using Lasserre-type moment/sum-of-squares (SOS)~\cite{Lasserre2001,Parrilo2003}.
Although these hierarchies provide a systematic framework 
for proving tightness, their size grows rapidly with the relaxation order, which limits their practical scalability. 
By contrast, the present paper incorporates doubly nonnegative constraints on the completion matrix and 
studies simple conic relaxations, thereby preserving tractability and avoiding the rapid dimensional growth 
associated with higher-order Lasserre/SOS relaxations.
The second strand concerns recovery from noisy observations~\cite{Candes2010survey}. 
In that setting, the objective is to estimate a low-rank matrix that explains the observed matrix and remains close to the observed data,
    typically under random-sampling~\cite{Candes2009,Candes2010,Recht2011simpler} and probabilistic-noise assumptions~\cite{Candes2010survey,Keshavan2009matrix}.
This strand of research has received extensive attention, motivated by its importance
in machine learning, collaborative filtering~\cite{Srebro2004maximum}, and data analysis~\cite{Mazumder2010spectral}.
Much of that literature emphasizes 
stable or high-probability recovery~\cite{Candes2009,Candes2010}
whereas our focus is on deterministic exact recovery through the tightness of tractable conic relaxations.

The main contributions of this paper are as follows.
\begin{enumerate}
\item We formulate rank-one DNN matrix completion as a QCQP and study its SDP, DNN, and CP relaxations from the viewpoint of tight conic relaxation.
\item For sparsity patterns whose blocks consist of cycles and edges, we prove that the dual formulations associated with the DNN and CP relaxations are equivalent.
\item For cycle-type sparsity patterns, we derive explicit sufficient conditions for tightness of the SDP and DNN relaxations, including local ratio bounds and cumulative-difference conditions. We also show that adding suitable edges to the sparsity pattern relaxes the ratio conditions required for tightness.
\end{enumerate}

The rest of this paper is organized as follows.
Section~\ref{sec:preliminary} introduces the conic framework, the QCQP formulation of rank-one DNN matrix completion, and the corresponding relaxations.
Section~\ref{sec:equivalence} presents the equivalence result for the dual relaxations.
Section~\ref{sec:tightness_for_conic} presents sufficient conditions for tightness for cycle-type sparsity patterns.
We conclude in Section~\ref{sec:conclusion}.





\section{Preliminaries} \label{sec:preliminary}

\subsection{Notation}

We write $\Real^s$ for the $s$-dimensional Euclidean space, and
$\Real_+^s$ and $\Real_-^s$ for  its nonnegative and nonpositive orthants, respectively.
When $s = 1$, the superscript is omitted.
Let $\0 \in \Real^s$ and $\1 \in \Real^s$ be the zero vector and the all-ones vector, respectively.
The symbol $\e^i$ denotes the $i$th unit vector.
The $i$th element of a vector $\v \in \Real^s$ is denoted by $v_i$.
For column vectors $\x^1,\ldots,\x^k$, the symbol $[\x^1;\ldots;\x^k]$ denotes their vertical concatenation in this order,
\textit{i.e.}, $[\x^1;\ldots;\x^k] \coloneqq \trans{[\trans{(\x^1)}, \ldots, \trans{(\x^k)}]}$.
For a given index set $\Omega = \left\{(i_1,j_1),\ldots,(i_s,j_s)\right\}$,
    we may define a vector $\boldsymbol{\alpha} = \trans{[\alpha_{i_1,j_1},\ldots,\alpha_{i_s,j_s}]} \in \Real^s$
    consisting only of the components indexed by $\Omega$. 
We denote the set of all $s \times s$ symmetric matrices by $\SymMat^s$. 
The symbols $O$ and $I$ are used to denote the zero matrix and the identity matrix, respectively. 
The matrix $E_{ij}$ is defined by $E_{ij} \coloneqq \e^{i}\trans{\left(\e^{i}\right)}$ if $i = j$,
    and $E_{ij} \coloneqq \e^{i}\trans{\left(\e^{j}\right)} + \e^{j}\trans{\left(\e^{i}\right)}$ otherwise.
For $d_1,\ldots,d_s \in \Real$, $\diag(d_1,\ldots,d_s)$ denotes the diagonal matrix with diagonal elements $d_1,\ldots,d_s$.
For a symmetric matrix $A \in \SymMat^s$ and index sets $I, J \subseteq \{1,\ldots,s\}$, 
    we denote by $A_{I,J}$ the submatrix of $A$ obtained by selecting rows indexed by $I$ and columns indexed by $J$.
When $I = J$, the principal submatrix $A_{I,I}$ is denoted by $A_{I}$.
For a symmetric matrix $A \in \SymMat^s$, the relation $A \succeq O$ means that $A$ is positive semidefinite (PSD).
We use the following convex cones of symmetric matrices~\cite{ShakedMonderer2021copositive}:
\begin{itemize}
\item the positive semidefinite (PSD) cone
    $\SymMat_+^s = \left(\SymMat_+^s\right)^* = \left\{ M \in \SymMat^s \,|\, \trans{\x}M\x \geq 0 \; \text{for all $\x \in \Real^s$} \right\}$,
\item the nonnegative cone
    $\SymN^s = \left(\SymN^s\right)^* = \left\{ M \in \SymMat^s \,|\, M_{ij} \geq 0 \; \text{for all $(i, j) \in \left\{1,\ldots,s\right\}^2$} \right\}$,
\item the doubly nonnegative (DNN) cone
    $\DNN^s = \SymMat_+^s \cap \SymN^s$,
\item the semidefinte-plus-nonnegative (SPN) cone
    $\SPN^s = \left(\DNN^s\right)^* = \SymMat_+^s + \SymN^s$,
\item the copositive (COP) cone
    $\COP^s = \left\{ M \in \SymMat^s \,|\, \trans{\x}M\x \geq 0 \; \text{for all $\x \in \Real_+^s$} \right\}$, and
\item the completely positive (CP) cone
    \[ \CP^s = \left\{ \sum_{i=1}^t \a^i\trans{\left(\a^i\right)} \;\middle|\;
        t \in \Natural, \; \a^i \in \Real_+^s \; \text{for all $i \in \left\{1,\ldots,t\right\}$} \right\}, \]
\end{itemize}
where $\left(\KC\right)^*$ denotes the dual cone of $\KC$.
For every $s \in \Natural$,
    the inclusions $\CP^s \subseteq \DNN^s \subseteq \SymMat_+^s \subseteq \SPN^s \subseteq \COP^s$ hold.
For $A, B \in \SymMat^s$,
    $\ip{A}{B}$ denotes the Frobenius inner product, i.e., $\ip{A}{B} = \trace(AB)$.

Let $G = (V,E)$ be an undirected graph. A graph $G$ is bipartite if its vertex set can be partitioned 
into two disjoint sets $T$ and $U$ such that no edge joins two vertices in the same set. 
Equivalently, every cycle in $G$ has even length. In that case, we write $G = (T,U,E)$ with $E \subseteq T \times U$.

\subsection{Sparsity pattern of matrices and results on cones}

The sparsity pattern of matrices in conic programming plays an important role in this paper.
We therefore begin by introducing the graph representation of the sparsity pattern of a symmetric matrix.

\begin{definition}
    The set $\overline{E} \subseteq \left\{1,\ldots,n\right\}^2$ is called 
        the sparsity pattern of a symmetric matrix $M \in \SymMat^n$ if
    \[ (i,j) \in \overline{E} \iff M_{ij} \neq 0 \quad \text{for all $(i, j) \in \{1,\ldots,n\}^2$}. \]
    Define the vertex set $V \coloneqq \left\{1,\ldots,n\right\}$
        and the edge set $E \coloneqq \overline{E} \setminus \left\{(i,i) \,\middle|\, i \in \left\{1,\ldots,n\right\}\right\}$.
    Then, $G = (V, E)$ is called the sparsity pattern graph of $M$.
\end{definition}
\noindent
The symmetry of  $M$  identifies the edges
$(i,j)$ and $(j,i)$ in the sparsity pattern graph.

We also recall the notion of a block.  
A cut vertex of a graph $G$ is a vertex whose removal increases the number of connected components in $G$.
A block of $G$ is a maximal subgraph with no cut vertex. 
When every block of the sparsity pattern graph $G$ of a symmetric matrix $M$ is either an edge or a cycle,
    membership in the SPN cone is equivalent to membership in the COP cone. 
This fact will be used in Section~\ref{sec:equivalence}.
\begin{lemma}[{\cite[Theorem~9.9]{shakedmonderer2016spn}}] \label{lem:spn_graph}
    Let $M \in \SymMat^s$, and let $G$ be its sparsity pattern graph.
    Suppose that each block of $G$ is either an edge or a cycle.
    Then, $M \in \SPN^s$ if and only if $M \in \COP^s$.
\end{lemma}

\subsection{Tight conic relaxations of QCQPs and strong duality} \label{ssec:sdp_relaxation}

We consider the following standard inequality form of QCQPs:  
\begin{equation} \label{eq:general_qcqp}
    \begin{array}{rl}
        \min\limits_{\x} & \trans{\x}Q^0\x + 2\trans{(\q^0)}\x \\
        \subto           & \trans{\x}Q^k\x + 2\trans{(\q^k)}\x \leq b_k, \quad k = 1,\ldots,m. 
    \end{array}
\end{equation}
Here, $\x \in \Real^n$ is a variable, while
$Q^0,\ldots,Q^m \in \SymMat^n$, $\q^0,\ldots,\q^m \in \Real^n$, and $\b \in \Real^m$ are given data.
In general, problem~\eqref{eq:general_qcqp} is nonconvex and difficult to solve globally 
unless the matrices  $Q^k$
are positive semidefinite for all $k \in \{0,\ldots,m\}$.
To obtain a tractable convex relaxation, 
we lift $\x\trans{\x}$ to a matrix variable and replace the implicit rank-one requirement in 
\eqref{eq:general_qcqp}
 by a conic constraint over a convex cone $\KC$, which yields

\begin{equation} \label{eq:general_sdp_primal}
    \begin{array}{rl}
        \min\limits_{\x, X} & \ip{Q^0}{X} + 2\trans{(\q^0)}\x \\
        \subto
                            & \ip{Q^k}{X} + 2\trans{(\q^k)}\x \leq b_k, \quad k = 1,\ldots,m, \\
                            & \begin{bmatrix} 1 & \trans{\x} \\ \x & X \end{bmatrix} \in \KC.
    \end{array}
\end{equation}
With an appropriate choice of  $\KC$ that makes \eqref{eq:general_sdp_primal} 
efficiently solvable, problem \eqref{eq:general_sdp_primal} yields a tractable lower bound for 
\eqref{eq:general_qcqp}.
The semidefinite programming (SDP) relaxation corresponds to \eqref{eq:general_sdp_primal} 
with $\KC = \SymMat_+^{n+1}$. The doubly nonnegative (DNN) relaxation is obtained by taking $\KC = \DNN^{n+1}$. 
When $X$ is elementwise nonnegative, the conditions $X \in \SymMat_+^n$ and $X \in \DNN^n$ are equivalent.
The dual problem of \eqref{eq:general_sdp_primal} is
\begin{equation} \label{eq:general_sdp_dual}
    \begin{array}{rl}
        \max\limits_{\boldsymbol{\xi}, \psi} & -\trans{\b}\boldsymbol{\xi} + \psi \\
        \subto
            & \begin{bmatrix} -\psi & \trans{(\q^0)} \\ \q^0 & Q^0 \end{bmatrix} 
                + \sum\limits_{k = 1}^m \xi_k \begin{bmatrix} 0 & \trans{(\q^k)} \\ \q^k & Q^k \end{bmatrix} \in \KC^*, \\
            & \boldsymbol{\xi} \geq \0.
    \end{array}
\end{equation}

A relaxation of \eqref{eq:general_qcqp} is called tight if its optimal value coincides with that of \eqref{eq:general_qcqp}.
A well-known sufficient condition for tightness is that \eqref{eq:general_sdp_primal} admits a rank-$1$ optimal solution of the form: 
\[
    \begin{bmatrix} 1 & \trans{\left(\x^*\right)} \\ \x^* & X^* \end{bmatrix}.
\]
In this case, the relaxation not only attains the optimal value of \eqref{eq:general_qcqp} but also recovers an optimal solution of \eqref{eq:general_qcqp}.
For this reason, the study of tight convex relaxations for QCQPs has been
    actively developed in recent years~\cite{ai2025tightness,Argue2020necessary,Azuma2021,Azuma2022,Kilinckarzan2021exactness,Kojima2025extending,Wang2021tightness}.
\subsection{Doubly nonnegative (DNN) completion problem}

A matrix completion problem is called a doubly nonnegative (DNN) completion problem  
    if the completed matrix is required to be doubly nonnegative, 
    i.e., a positive semidefinite matrix and elementwise nonnegative.
A rank-one DNN completion problem is a DNN completion problem with the rank-one constraint: 
\begin{equation} \label{eq:dnn_completion}
    \begin{array}{rl}
        \mathrm{find}    & X \in \Real^{m \times n} \\
        \subto           & A_{ij} = X_{ij}, \quad (i, j) \in \Omega, \\
                         & \rank X \leq 1, \quad X \geq O.
    \end{array}
\end{equation}
Since the completed matrix is required to be doubly nonnegative,  each specified element of $A$ must be nonnegative;   
    otherwise, the constraints $A_{ij} = X_{ij}$ and $X \geq O$ imply that \eqref{eq:dnn_completion} is infeasible. 
Since a rank-one matrix $X$ can be written as $\x\trans{\y}$ for some $\x \in \Real^m$ and $\y \in \Real^n$,  
    the $(i,j)$th element $X_{ij}$ in \eqref{eq:dnn_completion} equals $x_i y_j$.
Letting  $\z \coloneqq [\x; \y] \in \Real^{m+n}$,
    \eqref{eq:dnn_completion} can be reformulated as the following QCQP:  
\begin{equation} \label{eq:precise_qcqpform_dnn_completion}
    \begin{array}{rll}
        \mathrm{min}     & \sum_{i=1}^{m+n} z_i^2 \\
        \subto           & A_{ij} - z_i z_{m+j} = 0, & (i, j) \in \Omega, \\
        & \sum_{i=1}^{m+n} z_i^2 \leq r, \quad z_1 = 1, \quad \z \geq \0,
    \end{array}
\end{equation}
where $r$ is a sufficiently large constant.
By adding $z_iz_{m+j}$ to the objective  so that the equality constraints are relaxed,   
   we obtain the following alternative formulation:               
\begin{equation} \label{eq:qcqpform_dnn_completion}
    \begin{array}{rll}
        \mathrm{min}     & \sum_{i=1}^{m+n} z_i^2 + 2\rho \sum_{(i,j) \in \Omega}  z_i z_{m+j} \\
        \subto           & A_{ij} - z_i z_{m+j} \leq 0, & (i, j) \in \Omega, \\
                         & \sum_{i=1}^{m+n} z_i^2 \leq r, \quad z_1 = 1, \quad \z \geq \0,
    \end{array}
\end{equation}
where $\rho$ is a sufficiently large positive number.
In this paper, we study \eqref{eq:qcqpform_dnn_completion} rather than \eqref{eq:precise_qcqpform_dnn_completion}, 
    under the following natural assumption. 
\begin{assum} \label{asm:connected}
    The bipartite graph $(\{1,\ldots,m\},\{1,\ldots,n\}, \Omega)$ is connected and $A_{ij} > 0$ for all $(i,j) \in \Omega$.
\end{assum}
\noindent
When the graph is disconnected, the completion problem~\eqref{eq:dnn_completion} 
 decomposes into independent subproblems corresponding to the connected components of the graph. 
 We may therefore impose Assumption~\ref{asm:connected} without loss of generality and 
 focus on the connected case throughout the paper. 
 Under this assumption, any optimal solution $\z^*$ of \eqref{eq:qcqpform_dnn_completion} 
 has no zero components.

Applying the framework in Section~\ref{ssec:sdp_relaxation} to both~\eqref{eq:precise_qcqpform_dnn_completion} and~\eqref{eq:qcqpform_dnn_completion},
    we write their relaxations in a unified form.
Let \(\theta \in \{0,1\}\), and define
\[
    Q_\theta \coloneqq I + \theta \rho \sum_{(i,j)\in\Omega} E_{i,m+j}, \qquad
    \trianglelefteq_\theta \coloneqq \begin{cases}
        =    & \text{if $\theta=0$},\\
        \leq & \text{if $\theta=1$}.
    \end{cases}
\]
Then, we obtain the following conic relaxation: 
\begin{equation} \label{eq:dnn_relaxation}
    \begin{array}{rll}
        \mathrm{min}     & \ip{Q_\theta}{Z} \\
        \subto           & \ip{\left(2A_{ij} E_{11} - E_{i,m+j}\right)}{Z} \trianglelefteq_\theta 0, & (i, j) \in \Omega, \\
                         & \ip{\left(I - rE_{11}\right)}{Z} \leq 0, \\
                         & Z \in \KC, \quad Z_{11} = 1, \quad \ip{E_{1j}}{Z} \geq 0, \quad & j \in \{2,\ldots,m+n\},
    \end{array}
\end{equation}
where $\theta=0$ and $\theta=1$ correspond to the relaxation of \eqref{eq:precise_qcqpform_dnn_completion} and \eqref{eq:qcqpform_dnn_completion}, respectively,
    and $\KC$ is taken to be any one of $\SymMat_+^{m+n}$, $\DNN^{m+n}$ or $\CP^{m+n}$.
The constraints $\ip{E_{1j}}{Z} \geq 0$ encode the nonnegativity constraints $z_j \geq 0$ in the lifted space.
Let $\z^* \in \Real_+^{m+n}$ be an optimal solution of \eqref{eq:precise_qcqpform_dnn_completion} and define $Z^* \coloneqq \z^*\trans{\left(\z^*\right)}$.
Since \eqref{eq:dnn_relaxation} is a relaxation of \eqref{eq:precise_qcqpform_dnn_completion},
    $Z^*$ is a feasible solution of \eqref{eq:dnn_relaxation} with $\theta = 0$.
Introducing dual variables $\mu_0, \lambda, \boldsymbol{\w}$, 
    and setting the dual variables for the inactive nonnegativity constraints $E_{ij}\bullet Z\geq 0$ to zero,
    we write the dual problem of the relaxation as
\begin{equation} \label{eq:spn_relaxation}
    \begin{array}{rl}
        \max\limits_{\mu, \lambda, \w}
               & - \mu_0 \\
        \subto & (1 + \lambda)I + \sum\limits_{(i,j) \in \Omega} \left(\theta \rho - w_{i,j}\right) E_{i,m+j} + \left(\mu_0 - r\lambda + \sum\limits_{(i,j) \in \Omega} 2A_{ij} w_{i,j}\right)E_{11} \in \KC^*, \\
               & \mu_0 \in \Real, \quad \lambda \in \Real_+, \quad \theta w_{i,j} \geq 0, \quad (i, j) \in \Omega.
    \end{array}
\end{equation}

The relaxation~\eqref{eq:dnn_relaxation} satisfies Slater's condition~\cite{Boyd2004}
    when $\KC \in \left\{\SymMat_+^{m+n}, \DNN^{m+n}\right\}$ and $\theta \in \{0, 1\}$.
Define a matrix $P \in \SymMat^{m+n}$ given by
\[
    P_{\bar{i}\,\bar{j}} = \begin{cases}
        d & \text{if $\bar{i} = \bar{j} \in \{2,\ldots,m+n\}$}, \\
        \theta & \text{if $(\bar{i}, \bar{j}) \in E_\Omega$ or $(\bar{j}, \bar{i}) \in E_\Omega$} \\
        0 & \text{otherwise},
    \end{cases},
\]
where $d$ is a positive number and $E_\Omega \coloneqq \{(i, m+j), (m+j, i) \,|\, (i, j) \in \Omega\}$.
We will show that $Z^* + P$ is a strictly feasible solution of \eqref{eq:dnn_relaxation} for sufficiently large $d$.
When $\theta = 1$, for any $(i,j) \in \Omega$, it holds that
\[
    \ip{\left(2A_{ij} E_{11} - E_{i,m+j}\right)}{\left(Z^* + P\right)}
    < \ip{\left(2A_{ij} E_{11} - E_{i,m+j}\right)}{Z^*} \leq 0.
\]
When $\theta = 0$, the same constraints are equalities and are preserved at $Z^* + P$.
The $(1,1)$th element of $Z^* + P$ remains $1$.
By the Schur complement,
    the positive definiteness of $Z^* + P$ follows once we show that the following matrix is positive definite:
\begin{equation} \label{eq:schur_complement_z_plus_p}
    Z^*_{\{2,\ldots,m+n\}} + P_{\{2,\ldots,m+n\}}
    - \left(Z^*_{\{2,\ldots,m+n\},\{1\}} + P_{\{2,\ldots,m+n\},\{1\}}\right)
    \trans{\left(Z^*_{\{2,\ldots,m+n\},\{1\}} + P_{\{2,\ldots,m+n\},\{1\}}\right)}.
\end{equation}
Since all the diagonal elements of the principal submatrix $P_{\{2,\ldots,m+n\}}$ are equal to $d$,
taking $d$ sufficiently large ensures that \eqref{eq:schur_complement_z_plus_p}  is positive definite.
Moreover, since $(Z^*+P)_{11}=1$, choosing $r$ sufficiently large gives 
$\ip{\left(I - rE_{11}\right)}{\left(Z^* + P\right)} < 0$.
Hence, 
    Slater's condition holds.
Therefore, the optimal values of \eqref{eq:dnn_relaxation} and \eqref{eq:spn_relaxation} coincide,
    and any optimal primal-dual pair satisfies the complementary slackness condition.




\section{Equivalence of SPN and COP relaxations} \label{sec:equivalence}

In this section, we examine the equivalence of 
the dual relaxations associated with rank-one DNN matrix completion. 
In particular, when every block of the sparsity pattern graph is either an edge or a cycle, 
the dual problems \eqref{eq:spn_relaxation} with $\KC^*=\SPN^{m+n}$ and $\KC^*=\COP^{m+n}$ have the same optimal value. 
This result clarifies the relationship between these two convex-cone relaxations.

For notational convenience, define
\begin{equation} \label{eq:dual_matrix_in_spn_relaxation_with_w}  
    S_\theta(\mu_0, \lambda, \w) \coloneqq
        (1 + \lambda)I + \sum\limits_{(i,j) \in \Omega} \left(\theta\rho - w_{i,j}\right) E_{i,m+j}
        + \left(\mu_0 - r\lambda + \sum\limits_{(i,j) \in \Omega} 2A_{ij} w_{i,j}\right)E_{11},
\end{equation}
Then, the sparsity pattern graph of $S_\theta(\mu_0, \lambda, \w)$ is a subgraph of $(\left\{1,\ldots,m+n\right\}, \overline{\Omega})$, where
\[
    \overline{\Omega} \coloneqq
    \left\{ (i, i) \,\middle|\, i \in \left\{1,\ldots,m+n\right\}\right\} \cup
    \left[\bigcup_{(i, j) \in \Omega} \left\{ (i, m+j), (m+j, i) \right\}\right].
\]   
When $\widehat{\w} = \theta\rho\1$,  the second term in $S_\theta(\widehat{\mu}_0, \widehat{\lambda}, \widehat{\w})$ vanishes
    for any $\widehat{\mu}_0 \in \Real$ and $\widehat{\lambda} \in \Real_+$,
    so  $S_\theta(\widehat{\mu}_0, \widehat{\lambda}, \widehat{\w})$ is diagonal.
In particular, by taking $\widehat{\mu}_0 = -\sum_{(i,j) \in \Omega} 2\rho A_{ij}$ and $\widehat{\lambda} = 0$,
    then $S_\theta(\widehat{\mu}_0, \widehat{\lambda}, \widehat{\w}) = I \in \KC^*$ and
    the objective value of \eqref{eq:spn_relaxation} is $\sum_{(i,j) \in \Omega} 2\rho A_{ij}$.

The computational difficulty associated with the CP and COP cones is well-known.
For example, determining whether a given matrix is completely positive or copositive is NP-hard~\cite{dickinson2014computational,Murty1987some},
    and related tasks such as finding a CP-decomposition or determining the cp-rank are also difficult~\cite{dickinson2012linear}.
As a consequence, the problem~\eqref{eq:spn_relaxation} with $\KC^* = \COP^{m+n}$ is computationally intractable.
Compared with $\COP^{m+n}$,  the cone $\SPN^{m+n}$ is  more amenable to computation. 
The following theorem shows, however, that for the DNN completion problem,
    the two formulations~\eqref{eq:spn_relaxation} with $\KC^* = \SPN^{m+n}$ and $\KC^* = \COP^{m+n}$
    are equivalent.  
\begin{theorem} \label{thm:equivalence_of_relaxation}
    Let $G = \left(\left\{1,\ldots,m+n\right\}, \overline{\Omega}\right)$,
        and $\theta \in \{0, 1\}$.
    Suppose that
    \begin{enumerate*}[label={(\alph*)}]
    \item $G$ is connected, and
    \item every block in $G$ is either an edge or a cycle.
    \end{enumerate*}
    The dual relaxations~\eqref{eq:spn_relaxation} with $\KC^* = \SPN^{m+n}$ and $\KC^* = \COP^{m+n}$ are equivalent.  
 \end{theorem}
\begin{proof}
  Let $(\overline{\mu}_0,\overline{\lambda},\overline{\w})$ be a feasible point of \eqref{eq:spn_relaxation} 
  with $\KC^*=\COP^{m+n}$, and let $H$ be the sparsity pattern graph of 
  $S_\theta(\overline{\mu}_0,\overline{\lambda},\overline{\w})$. By construction, $H$ is a subgraph of $G=(\{1,\ldots,m+n\},\overline{\Omega})$. 
  Hence every block of $H$ is either an edge or a cycle. By Lemma~\ref{lem:spn_graph}, it follows 
  that $S_\theta(\overline{\mu}_0,\overline{\lambda},\overline{\w}) \in \SPN^{m+n}$. 
  Therefore, $(\overline{\mu}_0,\overline{\lambda},\overline{\w})$ is also feasible 
  for \eqref{eq:spn_relaxation} with $\KC^*=\SPN^{m+n}$.
\end{proof}
\noindent
Assumption~(b) in Theorem~\ref{thm:equivalence_of_relaxation} is not restrictive in many matrix-completion settings. 
Indeed,  from the viewpoint of matrix completion, 
edges that can be removed without destroying connectivity correspond 
to redundant constraints. In such cases, removing redundant edges may lead to a graph that satisfies 
the assumptions of Theorem~\ref{thm:equivalence_of_relaxation}.

\begin{rema}
   The result of Theorem~\ref{thm:equivalence_of_relaxation} holds independently of the constraint $\ip{\left(I - rE_{11}\right)}{Z} \leq 0$.
    Therefore, the CP and SPN relaxations are equivalent even if the original QCQP lacks the corresponding constraint $\sum_{i=1}^{m+n} z_i^2 \leq r$.
   Under this constraint, strong duality implies that each primal relaxation and its dual have the same optimal value. 
    If either primal relaxation is tight, then it admits a rank-one optimal solution $Z^* \coloneqq \z^*\trans{(\z^*)}$, 
    which is also feasible for the SDP relaxation. Hence the SDP relaxation is tight as well. Moreover, 
    if the SDP relaxation, i.e., \eqref{eq:dnn_relaxation} with $\KC=\SymMat_+^{m+n}$, is tight, 
    then it shares the same optimal value as the DNN and CP relaxations.
\end{rema}




\section{Tightness conditions for conic relaxations} \label{sec:tightness_for_conic}

In this section, we derive sufficient conditions under which the SDP and DNN relaxations~\eqref{eq:dnn_relaxation} are tight. 
Throughout, we focus on square completion problems with connected bipartite graphs.
Under these assumptions, 
tightness can be verified by constructing a feasible solution satisfying both a linear system
obtained from complementary slackness and the conic condition that the associated slack matrix lies in $\KC^*$.

Let $\z^*$ be a unique solution of \eqref{eq:precise_qcqpform_dnn_completion}, and define 
\[ Z^* \coloneqq \z^*\trans{\left(\z^*\right)}. \]
Then the columns of $Z^*$ are given by $\z^*, z^*_2\z^*, \ldots, z^*_{m+n}\z^*$.
By the Karush-Kuhn-Tucker conditions for \eqref{eq:dnn_relaxation} and \eqref{eq:spn_relaxation},   
    $Z^*$ is an optimal solution of \eqref{eq:dnn_relaxation} with $\theta = 0$
    if there exists a feasible point $(\mu_0, \lambda, \w)$ of \eqref{eq:spn_relaxation} such that   
\begin{align}
    O
    &= Z^*\left[
        I + \sum_{(i,j) \in \Omega} w_{i,j} \left(2A_{ij} E_{1,1} - E_{i,m+j}\right)
            + \mu_0 E_{11} + \lambda \left(I - rE_{11}\right) 
            \right] \nonumber \\
    &= Z^*\left[
        \left(1 + \lambda\right)I - \sum_{(i,j) \in \Omega} w_{i,j} E_{i,m+j} + \left(\mu_0 - r \lambda + \sum_{(i,j) \in \Omega} 2A_{ij} w_{i,j}\right)E_{11} \right] \nonumber \\
    &= \left(1 + \lambda\right) \begin{bmatrix} \z^* & z^*_2 \z^* & \cdots & z^*_{m+n}\z^* \end{bmatrix}
        - \sum_{(i,j) \in \Omega} w_{i,j}
            \begin{bNiceMatrix}[nullify-dots,first-row,code-for-first-row=\scriptstyle,margin]
                \RowStyle{\rotate}
                & & (i & & & & (m+j & & \\
                \Cdots & \0 & z^*_{m+j}\z^* & \0 & \Cdots & \0 & z^*_i\z^* & \0 & \Cdots
            \end{bNiceMatrix} \nonumber \\
    &\qquad 
        + \left(\mu_0 - r\lambda + \sum_{(i,j) \in \Omega} 2A_{ij} w_{i,j}\right) \begin{bmatrix} \z^* & \0 & \cdots & \0 \end{bmatrix}. 
        \label{eq:kkt_condition_for_zhat}
\end{align}
The last equality follows from $z^*_1 = 1$.   
Comparing the columns in \eqref{eq:kkt_condition_for_zhat}, we obtain the linear system 
\begin{subequations} \label{eq:linear_system_for_feasibility}
    \begin{empheq}[left = {\empheqlbrace \;}, right = {}]{alignat=4}
        &\sum_{(1, j) \in \Omega} \alpha_{1,j} z^*_{m+j} & &+ \beta & &= -\left(1 + \lambda\right), \label{eq:linear_system_for_feasibility_1} \\
        &\sum_{(i, j) \in \Omega} \alpha_{i,j} z^*_{m+j} & &        & &= -\left(1 + \lambda\right)z^*_i, & \quad & i = 2,\ldots,m, \label{eq:linear_system_for_feasibility_i} \\
        &\sum_{(i, j) \in \Omega} \alpha_{i,j} z^*_i     & &        & &= -\left(1 + \lambda\right)z^*_{m+j}, & \quad & j = 1,\ldots,n,\label{eq:linear_system_for_feasibility_mpj} \\
        & && && \alpha_{i,j} \in \Real, \; (i,j) \in \Omega, & \quad & \beta \in \Real,\; \lambda \in \Real_+, \label{eq:linear_system_for_feasibility_domain}
    \end{empheq}
\end{subequations}
where $\alpha_{i,j} \coloneqq - w_{i,j}$ and $\beta \coloneqq \mu_0 - r\lambda + \sum_{(i,j) \in \Omega} 2A_{ij} w_{i,j}$.    
Under this change of variables, the matrix $S_0(\mu_0, \lambda, \w)$ defined in \eqref{eq:dual_matrix_in_spn_relaxation_with_w}
    can be equivalently written as
\[
    S_0(\boldsymbol{\alpha}, \beta, \lambda) \coloneqq
        (1 + \lambda)I + \sum\limits_{(i,j) \in \Omega} \alpha_{i,j} E_{i,m+j} + \beta E_{11}.
\]
Although the linear system~\eqref{eq:linear_system_for_feasibility} includes the variable $\lambda$, its role is inessential.
The following lemma shows that we may take $\lambda = 0$ without loss of generality.
\begin{lemma} \label{lem:meaning_of_lambda}
    System~\eqref{eq:linear_system_for_feasibility} admits a feasible solution $(\overline{\boldsymbol{\alpha}}, \overline{\beta}, \overline{\lambda})$
    satisfying $S_0(\overline{\boldsymbol{\alpha}}, \overline{\beta}, \overline{\lambda}) \in \KC^*$ 
      for some $\overline{\lambda} \geq 0$ 
    if and only if it admits a feasible solution $(\widehat{\boldsymbol{\alpha}},\widehat{\beta},0)$ satisfying $S_0(\widehat{\boldsymbol{\alpha}},\widehat{\beta},0) \in \KC^*$. 
    Moreover, if the former solution satisfies $\overline{\alpha}_{i,j}\le \rho$ for all $(i,j)\in\Omega$,
        then the latter solution can be chosen to satisfy $\widehat{\alpha}_{i,j}\le \rho$ for all $(i,j)\in\Omega$.
\end{lemma}
\begin{proof}
    The ``if'' direction is immediate, since $\overline{\lambda}=0$ is a special case of $\overline{\lambda} \geq 0$.  
    For the converse, 
        let $(\overline{\boldsymbol{\alpha}}, \overline{\beta}, \overline{\lambda})$ be
        a feasible solution of \eqref{eq:linear_system_for_feasibility} such that
        $S_0(\overline{\boldsymbol{\alpha}}, \overline{\beta}, \overline{\lambda}) \in \KC^*$.   
    Define   
    \[
        \widehat{\alpha}_{i,j} \coloneqq \frac{\overline{\alpha}_{i,j}}{1+\overline{\lambda}} \quad \text{for all $(i,j) \in \Omega$,}
        \qquad \widehat{\beta} \coloneqq \frac{\overline{\beta}}{1+\overline{\lambda}}.
    \]
    Since $1+\overline{\lambda} > 0$,  dividing \eqref{eq:linear_system_for_feasibility_1}--\eqref{eq:linear_system_for_feasibility_mpj}
        by $1+\overline{\lambda}$ shows that $(\widehat{\boldsymbol{\alpha}}, \widehat{\beta}, 0)$ satisfies \eqref{eq:linear_system_for_feasibility}.
            Moreover, if $\overline{\alpha}_{i,j}\le\rho$ for every $(i,j)\in\Omega$, then
    \[
        \widehat{\alpha}_{i,j} = \frac{\overline{\alpha}_{i,j}}{1+\overline{\lambda}} \le \frac{\rho}{1+\overline{\lambda}} \le \rho,
    \]
        and hence the bound is preserved.
    Finally,  
    \[
        S_0(\widehat{\boldsymbol{\alpha}}, \widehat{\beta}, 0)
            = I + \sum_{(i,j) \in \Omega} \widehat{\alpha}_{i,j}E_{i,m+j} + \widehat{\beta}E_{11}
            = \frac{1}{1+\overline{\lambda}} S_0(\overline{\boldsymbol{\alpha}}, \overline{\beta}, \overline{\lambda}) \in \KC^*.
    \]
    Thus, $(\widehat{\boldsymbol{\alpha}},\widehat{\beta},0)$ is a feasible solution of~\eqref{eq:linear_system_for_feasibility}
    whose slack matrix belongs to $\KC^*$.
\end{proof}
\noindent
Therefore, if a triple $(\overline{\boldsymbol{\alpha}},\overline{\beta},\overline{\lambda}) \in \Real^{|\Omega|}\times\Real\times\Real_+$ satisfies \eqref{eq:linear_system_for_feasibility}
and
$S_0(\overline{\boldsymbol{\alpha}},\overline{\beta},\overline{\lambda}) \in \KC^*$, then $Z^*$ is an optimal solution of \eqref{eq:dnn_relaxation}. 
This yields the following sufficient condition for tightness.

\begin{prop} \label{prop:tightness_on_system}
    Let $\z^*$ be an optimal solution of \eqref{eq:precise_qcqpform_dnn_completion},
        $Z^* \coloneqq \z^*\trans{\left(\z^*\right)}$,
        and let $\KC$ be either $\DNN^{m+n}$ or $\SymMat_+^{m+n}$. 
    Suppose there exists a feasible solution 
    $(\overline{\boldsymbol{\alpha}},\overline{\beta},\overline{\lambda})$
    of \eqref{eq:linear_system_for_feasibility} such that
    $S_0(\overline{\boldsymbol{\alpha}},\overline{\beta},\overline{\lambda}) \in \KC^*$.
    Then, the following properties hold:
    \begin{enumerate}[label={(\alph*)}]
    \item
        for the equality formulation~\eqref{eq:precise_qcqpform_dnn_completion},
        the relaxation~\eqref{eq:dnn_relaxation} with $\theta = 0$ admits $Z^*$ as an optimal solution and is tight;
    \item
        for the penalized inequality formulation~\eqref{eq:qcqpform_dnn_completion},
        if, in addition,
        \[
            \overline{\alpha}_{i,j} \le \rho \quad \forall (i,j)\in\Omega,
        \]
        then
        the relaxation~\eqref{eq:dnn_relaxation} with $\theta = 1$ also admits $Z^*$ as an optimal solution and is tight.
    \end{enumerate}
\end{prop}
\begin{proof}
    By Lemma~\ref{lem:meaning_of_lambda}, we may assume that $\overline{\lambda} = 0$ without loss of generality.
    In the proof of (b), the last assertion of Lemma~\ref{lem:meaning_of_lambda} ensures that
        this reduction can be made while preserving the bounds $\overline{\alpha}_{i,j} \leq \rho$ for all $(i,j) \in \Omega$.
    We first consider the primal and dual relaxations of~\eqref{eq:precise_qcqpform_dnn_completion},
        namely, \eqref{eq:dnn_relaxation} and \eqref{eq:spn_relaxation} with $\theta = 0$.
    Since $\z^*$ is feasible for \eqref{eq:precise_qcqpform_dnn_completion},
        $Z^*$ is feasible for its relaxation, namely, \eqref{eq:dnn_relaxation} with $\theta = 0$.
    Define
    \[
        \overline{w}_{i,j} \coloneqq -\overline{\alpha}_{i,j} \quad \text{for all } (i,j)\in\Omega, \qquad
        \overline{\mu}_0 \coloneqq \overline{\beta} + r\overline{\lambda}
            + \sum_{(i,j)\in\Omega}2A_{ij}\overline{\alpha}_{i,j}.
    \]
    Then, we have
    \[
        S_0(\overline{\mu_0},\overline{\lambda},\overline{\w})
        = S_0(\overline{\boldsymbol{\alpha}},\overline{\beta},\overline{\lambda}) \in \KC^*.
    \]
    Hence, $(\overline{\mu_0},\overline{\lambda},\overline{\w})$ is feasible for the dual relaxation~\eqref{eq:spn_relaxation}.
    By construction, the linear system~\eqref{eq:linear_system_for_feasibility} is precisely the columnwise form of
        $Z^*S_0(\overline{\boldsymbol{\alpha}},\overline{\beta},\overline{\lambda})=O$,
        and therefore $Z^*S_0(\overline{\mu_0},\overline{\lambda},\overline{\w})=O$.
    For every $(i,j) \in \Omega$,
        $\ip{\left(2A_{ij}E_{11} - E_{i,m+j}\right)}{Z^*} = 0$
        implies $\overline{w}_{i,j}\left[\ip{\left(2A_{ij}E_{11} - E_{i,m+j}\right)}{Z^*}\right] = 0$.
    In addition, it follows from $\overline{\lambda} = 0$
        that $\overline{\lambda}\left[\ip{\left( I - rE_{11} \right)}{Z^*}\right] = 0$.
    Thus, $Z^*$ and $(\overline{\mu_0},\overline{\lambda},\overline{\w})$ satisfy the complementary slackness condition for the conic constraint.
    Since Slater's condition holds for \eqref{eq:dnn_relaxation} by Section~\ref{ssec:sdp_relaxation},
        the KKT conditions are sufficient for optimality.
    Therefore, $Z^*$ is an optimal solution of \eqref{eq:dnn_relaxation}.
    Since $Z^*$ is of rank-one, \eqref{eq:dnn_relaxation} is tight for \eqref{eq:precise_qcqpform_dnn_completion}.

    We next show that $Z^*$ is also an optimal solution of the relaxation of \eqref{eq:qcqpform_dnn_completion},
        \textit{i.e.}, the case $\theta = 1$.
    Since the equality constraints of \eqref{eq:precise_qcqpform_dnn_completion} imply the corresponding inequalities of \eqref{eq:qcqpform_dnn_completion},
        $Z^*$ is feasible for \eqref{eq:dnn_relaxation} with $\theta=1$.
    Define
    \[
        \widehat{\mu}_0 \coloneqq \overline{\mu}_0 - \sum_{(i,j) \in \Omega} 2A_{ij}\rho, \qquad
        \widehat{\lambda} \coloneqq 0, \qquad
        \widehat{\w} \coloneqq \overline{\w} + \rho\1.
    \]
    Substituting them into $S_1$ yields
    \begin{align*}
        S_1(\widehat{\mu}_0, \widehat{\lambda}, \widehat{\w})
        &= I + \sum_{(i,j)\in\Omega}\left(\rho-\widehat{w}_{i,j}\right)E_{i,m+j}
            + \left(\widehat{\mu}_0 + \sum_{(i,j)\in\Omega}2A_{ij}\widehat{w}_{i,j}\right)E_{11} \\
        &= I - \sum_{(i,j)\in\Omega}\overline{w}_{i,j}E_{i,m+j}
            + \left(\overline{\mu}_0 + \sum_{(i,j)\in\Omega}2A_{ij}\overline{w}_{i,j}\right)E_{11} \\
        &= S_0(\overline{\mu}_0,0,\overline{\w}) \in \KC^*.
    \end{align*}
    Moreover, for every $(i,j)\in\Omega$, the additional condition gives
    \[
        \widehat{w}_{i,j}=\overline{w}_{i,j}+\rho=-\overline{\alpha}_{i,j}+\rho \ge 0.
    \]
    Thus, $(\widehat{\mu}_0,\widehat{\lambda},\widehat{\w})$ is feasible for \eqref{eq:spn_relaxation} with $\theta = 1$.
    Since $S_1(\widehat{\mu}_0,\widehat{\lambda},\widehat{\w})=S_0(\overline{\mu}_0,0,\overline{\w})$, the relation
    $Z^*S_1(\widehat{\mu}_0,\widehat{\lambda},\widehat{\w})=O$ holds.
    The remaining complementary slackness relations hold since the corresponding inequalities are active at $Z^*$ and $\widehat{\lambda}=0$.
    Therefore, the same KKT argument as above shows that $Z^*$ is an optimal solution of \eqref{eq:dnn_relaxation} with $\theta=1$.
\end{proof}

\noindent
By Proposition~\ref{prop:tightness_on_system}, for both $\theta = 0$ and $\theta = 1$,
    it suffices to consider the slack matrix $S_0$ rather than $S_1$ when proving the tightness of the relaxation~\eqref{eq:dnn_relaxation}.
Hence, throughout the remainder of this section, 
    we omit the subscript and write $S$ in place of $S_0$ whenever no confusion arises.

The following lemma will be used repeatedly
    in the subsequent tightness proofs to show that the slack matrix $S(\boldsymbol{\alpha}, \beta, \lambda)$ lies in $\KC^*$.
\begin{lemma} \label{lem:nonpositive_certificate}
    Suppose that $m = n$.
    Let $(\boldsymbol{\alpha},\beta,0)$ be a feasible point of \eqref{eq:linear_system_for_feasibility}.
    If $\boldsymbol{\alpha} \leq \0$, then $S(\boldsymbol{\alpha},\beta,0) \succeq O$.
\end{lemma}
\begin{proof}
    Let $D \coloneqq \diag(1+\beta,1,\dots,1)$ and define $B \in \Real^{n \times n}$ by
    \[
        B_{ij} = \begin{cases}
            \alpha_{i, j} &\text{if $(i,j) \in \Omega$}, \\
            0             &\text{otherwise}.
        \end{cases}
    \]
    Then, the slack matrix can be represented by a $2 \times 2$ block matrix, as follows:
    \[   
        S(\boldsymbol{\alpha},\beta,0) = \begin{bmatrix} D & B\\ \trans{B} & I \end{bmatrix}.
    \]
    Since the lower-right block is the identity matrix, the Schur complement yields
    \[
        S(\boldsymbol{\alpha},\beta,0) \succeq O \quad\iff\quad  M \coloneqq D - B\trans{B} \succeq O.
    \]
    Write $\z^* = [\x^*; \y^*] \in \Real^{2n}$ so that
        the linear system~\eqref{eq:linear_system_for_feasibility_n} is equivalent to
    \[
        \trans{B}\x^* = -\y^*, \qquad B\y^* = -D\x^*.
    \]
    Hence, we have
    \[
        M\x^* = D\x^* - B(\trans{B}\x^*) = D\x^* + B\y^* = 0.
    \]
    Since $\boldsymbol{\alpha} \leq \0$, the matrix $B$ is elementwise nonpositive.
    Hence, for every $i \neq j$, we have
    \[
        M_{ij} = - \sum_{k=1}^n B_{ik} B_{jk} \leq 0.
    \]
    Let $U \coloneqq \diag(x^*_1,\ldots,x^*_n)$.
    Since all elements of $\x^*$ are positive in the present setting, the matrix $U$ is nonsingular.
    Then, $L \coloneqq U M U \in \SymMat^n$ has nonpositive off-diagonal elements and satisfies
    \[
        L\1 = U M \x^* = 0.
    \]
    Thus, $L$ is a weighted Laplacian matrix of an undirected graph, and therefore $L \succeq O$ (see~\cite{Merris1994laplacian}).
    Since $M = U^{-1} L U^{-1}$, it follows that $M \succeq O$,
        and hence $S(\boldsymbol{\alpha},\beta,0) \succeq O$.
\end{proof}

\subsection{DNN completion problem with a cycle graph}

We first examine the case where the bipartite graph determined by  $\Omega$ is a cycle, 
as this is the fundamental sparsity pattern studied in this paper.
The cycle structure is sparse enough to permit an explicit analysis of the feasibility system, 
yet structured enough to make positive semidefiniteness of the slack matrix nontrivial.
We therefore begin with the smallest case to illustrate the main idea and then turn to the general $n \times n$ setting.  

A DNN matrix completion problem with $m = n = 2$ can be formulated as \eqref{eq:qcqpform_dnn_completion} with $\z \in \Real^4$.
In this example, we set $\Omega = \left\{ (1, 1), (1, 2), (2, 1), (2, 2) \right\}$.
Then, the bipartite graph $(\{1,2\},\{1,2\},\Omega)$ is a cycle.
The linear system~\eqref{eq:linear_system_for_feasibility} with $\lambda = 0$ becomes 
\[
    \begin{bmatrix}
        z^*_3 & z^*_4 & 0 & 0 & 1 \\
        0 & 0 & z^*_3 & z^*_4 & 0 \\
        z^*_1 & 0 & z^*_2 & 0 & 0 \\
        0 & z^*_1 & 0 & z^*_2 & 0
    \end{bmatrix}
    \begin{bmatrix} \alpha_{1,1} \\ \alpha_{1,2} \\ \alpha_{2,1} \\ \alpha_{2,2} \\ \beta \end{bmatrix} = -\z^*.
\]
After a suitable left multiplication, this system can be rewritten as
\[
    \begin{bmatrix}
        \frac{z^*_3}{z^*_1} & \frac{z^*_4}{z^*_1} & 0 & 0 & 1 \\
        0 & 0 & \frac{z^*_3}{z^*_2} & \frac{z^*_4}{z^*_2} & 0 \\
        0 & 0 & 0 & 0 & 1 \\
        0 & \frac{z^*_1}{z^*_4} & 0 & \frac{z^*_2}{z^*_4} & 0
    \end{bmatrix}
    \begin{bmatrix} \alpha_{1,1} \\ \alpha_{1,2} \\ \alpha_{2,1} \\ \alpha_{2,2} \\ \beta \end{bmatrix} =
    \begin{bmatrix} -1 \\ -1 \\ - \left(z^*_1\right)^2 - \left(z^*_2\right)^2 + \left(z^*_3\right)^2 + \left(z^*_4\right)^2 \\ -1 \end{bmatrix}.
\]  
Thus, $\beta = - \left(z^*_1\right)^2 - \left(z^*_2\right)^2 + \left(z^*_3\right)^2 + \left(z^*_4\right)^2$ is uniquely determined.
To analyze the system, suppose that $z^*_{2+j} / z^*_i \in \left[1, 1 + \delta\right]$ for all $(i, j) \in \Omega$ and $\delta >0$. 
Substituting the expression for $\beta$ into the first equation yields
\[
    \left(z^*_2\right)^2 - \left(z^*_3\right)^2 - \left(z^*_4\right)^2 = z^*_3\alpha_{1,1} + z^*_4\alpha_{1,2}.
\]
Since $z^*_3/z^*_2 \ge 1$ and $z^*_4/z^*_2 \ge 1$, the left-hand side is nonpositive. 
Therefore, the equation admits a solution with $\alpha_{1,1},\alpha_{1,2}\le \rho$.
Since $\alpha_{1,1}, \alpha_{1,2} \leq \rho$ and $z^*_3, z^*_4$ are nonnegative by assumption,
    the above equality admits a solution $(\alpha_{1,1},\alpha_{1,2})$.    
Once $\alpha_{1,1}$ and $\alpha_{1,2}$ are fixed, the remaining variables are determined accordingly.  
Indeed, we can take
\begin{align*}
    \alpha_{1,1} &\coloneqq -\frac{-(z^*_2)^2 + (z^*_3)^2 + (z^*_4)^2}{(z^*_3)^2 + (z^*_4)^2}z^*_3, &
    \alpha_{1,2} &\coloneqq -\frac{-(z^*_2)^2 + (z^*_3)^2 + (z^*_4)^2}{(z^*_3)^2 + (z^*_4)^2}z^*_4, \\
    \alpha_{2,1} &\coloneqq -\frac{z^*_2 z^*_3}{(z^*_3)^2 + (z^*_4)^2}, &
    \alpha_{2,2} &\coloneqq -\frac{z^*_2 z^*_4}{(z^*_3)^2 + (z^*_4)^2}.
\end{align*}
Note that all elements of $\boldsymbol{\alpha}$ are nonpositive.
In addition, $\boldsymbol{\alpha}$ must satisfy $S(\boldsymbol{\alpha},\beta,0)\succeq O$. The slack matrix can be written as
\[
    S(\alpha,\beta,0)=
    \begin{bNiceArray}{cc|cc}[margin]
        - (z^*_2)^2 +(z^*_3)^2 + (z^*_4)^2 & 0 & \alpha_{1,1} & \alpha_{1,2} \\
        0 & 1 & \alpha_{2,1} & \alpha_{2,2}\\
        \hline
        \alpha_{1,1} & \alpha_{2,1} & 1 & 0\\
        \alpha_{1,2} & \alpha_{2,2} & 0 & 1
    \end{bNiceArray} \eqqcolon \begin{bmatrix}
        D & B \\
        \trans{B} & I
    \end{bmatrix}.
\]  
By the Schur complement, $S(\alpha,\beta,0) \succeq O$ if and only if   
\begin{align*}
    O \preceq D - B\trans{B}
    &= \frac{- (z^*_2)^2 + (z^*_3)^2 + (z^*_4)^2}{(z^*_3)^2 + (z^*_4)^2}
        \begin{bmatrix} (z^*_2)^2 & -z^*_2 \\ -z^*_2 & 1 \end{bmatrix} \\
    &= \frac{- (z^*_2)^2 + (z^*_3)^2 + (z^*_4)^2}{(z^*_3)^2 + (z^*_4)^2}
        \begin{bmatrix} z^*_2 \\ -1 \end{bmatrix}\trans{\begin{bmatrix} z^*_2 \\ -1 \end{bmatrix}}.
\end{align*}
Under the assumption $z^*_3/z^*_2 \ge 1$ and $z^*_4/z^*_2 \ge 1$, we have
  $-(z^*_2)^2+(z^*_3)^2+(z^*_4)^2 \ge 0$.
  Hence $D-B\trans{B}\succeq O$, and therefore
  $S(\boldsymbol{\alpha},\beta,0)\succeq O$.
Therefore, Proposition~\ref{prop:tightness_on_system} establishes the tightness of \eqref{eq:dnn_relaxation} under both $\theta = 0$ and $\theta = 1$.
The following proposition summarizes the preceding argument. 
\begin{prop}
    Let $\z^*$ be an optimal solution of \eqref{eq:precise_qcqpform_dnn_completion}
        for a $2 \times 2$ DNN completion problem, \textit{i.e.}, $m = n = 2$,
        and $\KC$ be either $\SymMat_+^{4}$ or $\DNN^{4}$.
    Suppose that $\Omega = \left\{ (1, 1), (1, 2), (2, 1), (2, 2) \right\}$.
    If $z^*_{2+j} / z^*_i \geq 1$ for all $(i, j) \in \Omega$,
        the relaxation~\eqref{eq:dnn_relaxation} admits $\z^*\trans{\left(\z^*\right)}$ as a rank-one optimal solution
        for each $\theta \in \{0, 1\}$.
    Consequently, the relaxations of \eqref{eq:precise_qcqpform_dnn_completion} and \eqref{eq:qcqpform_dnn_completion} are both tight.
\end{prop}  

We now extend the preceding example to the $n \times n$ DNN completion problem associated with cycle sparsity. 
Since the problem is invariant under relabeling that preserves the cycle structure, it suffices to consider the representative pattern
\[
    \Omega = \left\{ (i, i) \,\middle|\, i \in \left\{1,\ldots,n\right\} \right\}
        \cup \left\{ (i, i-1) \,\middle|\, i \in \left\{2,\ldots,n\right\} \right\}
        \cup \left\{ (1, n) \right\} \eqqcolon \Omega_{\text{cycle}}.
\]      
The linear system corresponding to this choice of $\Omega$ takes the form  
\begin{subequations} \label{eq:linear_system_for_feasibility_n}
    \begin{empheq}[left = {\empheqlbrace \;}, right = {}]{alignat=5}
        &\alpha_{1,1}   z^*_{n+1}   & &+ \alpha_{1,n}   z^*_{2n}  & &+ \beta & &= -1, \label{eq:linear_system_for_feasibility_n_1} \\
        &\alpha_{i,i-1} z^*_{n+i-1} & &+ \alpha_{i,i}   z^*_{n+i} & &        & &= -z^*_i,     & \quad & i = 2,\ldots,n, \label{eq:linear_system_for_feasibility_n_i} \\
        &\alpha_{i,i}   z^*_{i}     & &+ \alpha_{i+1,i} z^*_{i+1} & &        & &= -z^*_{n+i}, & \quad & i = 1,\ldots,n-1, \label{eq:linear_system_for_feasibility_n_npi} \\
        &\alpha_{n,n}   z^*_{n}     & &+ \alpha_{1,n}   z^*_1     & &        & &= -z^*_{2n}, \label{eq:linear_system_for_feasibility_n_2n} \\
        & \boldsymbol{\alpha} \in \Real^{|\Omega|}, && \beta \in \Real.
    \end{empheq}
\end{subequations}

To derive sufficient conditions for tightness, we prove the following lemma using \eqref{eq:linear_system_for_feasibility_n}.  
\begin{lemma} \label{lem:the_others_except_beta}
    Let $(\boldsymbol{\alpha},\beta)$ be a feasible solution of \eqref{eq:linear_system_for_feasibility_n}.    
    Then, $\alpha_{1,1}$ and $\alpha_{1,n}$ satisfy
    \begin{equation} \label{eq:difference_of_squared_sum}
        - 1 + \sum_{i=1}^n \left(z^*_i\right)^2 - \sum_{i=1}^n \left(z^*_{n+i}\right)^2 = \alpha_{1,1} z^*_{n+1} + \alpha_{1,n} z^*_{2n}.
    \end{equation}
\end{lemma}
\begin{proof}
    Replacing $i$ by $i-1$ in \eqref{eq:linear_system_for_feasibility_n_npi} and multiplying by $z^*_{n+i-1}$ gives
    \[
        z^*_{n+i-1} \left(\alpha_{i-1,i-1} z^*_{i-1} + \alpha_{i,i-1} z^*_{i} \right) = -\left(z^*_{n+i-1}\right)^2
        \quad \text{for $i = 2,\ldots,n$.}
    \]
   Multiplying \eqref{eq:linear_system_for_feasibility_n_i} by $-z^*_i$   for each $i \in \{2,\ldots,n\}$ and adding the result yields
    \begin{equation} \label{eq:sum_of_second_and_forth}
        \alpha_{i-1,i-1} z^*_{i-1} z^*_{n+i-1} - \alpha_{i,i} z^*_i z^*_{n+i}
        = -\left(z^*_{n+i-1}\right)^2 + \left(z^*_i\right)^2.       
    \end{equation}
    It follows that
    \begin{align*}  
        \alpha_{1,1} z^*_{1} z^*_{n+1}
        &= \left(z^*_2\right)^2 - \left(z^*_{n+1}\right)^2 + \alpha_{2,2} z^*_2 z^*_{n+2} \\
        &= -\sum_{i=1}^{n-1} \left(z^*_{n+i}\right)^2 + \sum_{i=2}^{n} \left(z^*_i\right)^2 + \alpha_{n,n} z^*_n z^*_{2n}
            & &\quad  \text{(by repeatedly applying \eqref{eq:sum_of_second_and_forth})} \\
        &= -\sum_{i=1}^{n} \left(z^*_{n+i}\right)^2 + \sum_{i=2}^{n} \left(z^*_i\right)^2 - \alpha_{1,n} z^*_1 z^*_{2n} 
            & &\quad \text{(by \eqref{eq:linear_system_for_feasibility_n_2n})}
    \end{align*}   
    Adding $1 = \left(z^*_1\right)^2$ to the both side yields the desired identity. 
\end{proof}

The following theorem provides our first sufficient condition for the tightness of SDP and DNN relaxations.

\begin{theorem} \label{thm:upperbound_delta_cycle}
    Let $\z^*$ be an optimal solution of \eqref{eq:precise_qcqpform_dnn_completion} for an $n \times n$ DNN completion problem, 
        $\KC$ be either $\SymMat_+^{2n}$ or $\DNN^{2n}$,
        and $\Omega$ be the index set of its specified elements.
    Suppose that $\Omega = \Omega_{\text{cycle}}$.
    If $z^*_{n+j} / z^*_i \in \left[1, \sqrt{1 + \frac{1}{n-1}}\right]$ for all $(i, j) \in \Omega$,
        then the relaxation~\eqref{eq:dnn_relaxation} admits $\z^*\trans{\left(\z^*\right)}$ as a rank-one optimal solution for each $\theta \in \{0, 1\}$.
    Consequently, the relaxations of \eqref{eq:precise_qcqpform_dnn_completion} and \eqref{eq:qcqpform_dnn_completion} are both tight.
\end{theorem}
\begin{proof}
    By assumption, for all $t \in \{2,\ldots,n-1\}$, we have   
    \[
        \left(z^*_t\right)^2 \leq \left(z^*_{n+t}\right)^2 \leq \left(1+\frac{1}{n-1}\right) \left(z^*_t\right)^2,
        \quad \left(z^*_{t+1}\right)^2 \leq \left(z^*_{n+t}\right)^2 \leq \left(1+\frac{1}{n-1}\right)\left(z^*_{t+1}\right)^2.
    \]
    \[
        \Delta_1 \coloneqq 0, \qquad
        \Delta_k \coloneqq \sum_{t=2}^k \left[\left(z^*_{n+t}\right)^2-\left(z^*_t\right)^2\right] \geq 0 \quad (k=2,\ldots,n).
    \]
    We show that
    \begin{equation} \label{eq:upperbound_delta_cycle_induction}
        \Delta_k \leq \frac{k-1}{n-1}(z^*_{k+1})^2
    \end{equation}
    holds for all $k \in \{1,\ldots,n-1\}$ by induction.
    The case $k = 1$ is immediate from $\Delta_1 = 0$.
    Assume \eqref{eq:upperbound_delta_cycle_induction} holds for some $k \in \{1,\ldots,n-2\}$.
    It follows from the definition of $\Delta_{k+1}$ that $\Delta_{k+1} = \Delta_k + \left(z^*_{n+k+1}\right)^2 - \left(z^*_{k+1}\right)^2$.  
    If $\left(z^*_{k+1}\right)^2 \leq \left(z^*_{k+2}\right)^2$, then we have
    \[
        \Delta_{k+1}
        \leq \frac{k-1}{n-1}\left(z^*_{k+1}\right)^2 + \left(1+\frac{1}{n-1}\right)\left(z^*_{k+1}\right)^2 - \left(z^*_{k+1}\right)^2
        =    \frac{k}{n-1}\left(z^*_{k+1}\right)^2
        \leq \frac{k}{n-1}\left(z^*_{k+2}\right)^2.
    \]
    Otherwise, we have
    \begin{align*}
        \Delta_{k+1}
        &\leq \frac{k-1}{n-1}\left(z^*_{k+1}\right)^2 + \left(1+\frac{1}{n-1}\right)\left(z^*_{k+2}\right)^2 - \left(z^*_{k+1}\right)^2 \\
        &\leq \left(1+\frac{1}{n-1}\right)\left(z^*_{k+2}\right)^2 - \left(1 - \frac{k-1}{n-1}\right)\left(z^*_{k+2}\right)^2
            = \frac{k}{n-1}\left(z^*_{k+2}\right)^2.
    \end{align*}
    This proves the induction step.
        Define $\boldsymbol{\alpha}$ by   
    \begin{alignat*}{2}
        \alpha_{1,1}   &\coloneqq -\frac{z_{n+1}^*}{z_1^*}, \qquad&
        \alpha_{1,n}   &\coloneqq -\frac{\Delta_n}{z_1^* z_{2n}^*}, \\
        \alpha_{i,i-1} &\coloneqq -\frac{\Delta_{i-1}}{z_i^* z_{n+i-1}^*}, \qquad&
        \alpha_{i,i}   &\coloneqq -\frac{(z_i^*)^2 - \Delta_{i-1}}{z_i^* z_{n+i}^*} \quad (i=2,\ldots,n),
    \end{alignat*}
    and $\beta \coloneqq - \sum_{i = 1}^n \left(z^*_i\right)^2 + \sum_{j = 1}^n \left(z^*_{n+j}\right)^2$.   
    Then, $\alpha_{1,1} \leq 0$ holds by assumption,
        while the inequalities $\alpha_{1,n} \leq 0$ and $\alpha_{i,i-1} \leq 0 \, (i \in \{2,\ldots,n\})$
        also follow from the nonnegativity of $\Delta_k$.
    By \eqref{eq:upperbound_delta_cycle_induction} with $k=i-1$, we have
     $\Delta_{i-1} \le \frac{i-2}{n-1}(z_i^*)^2 \le (z_i^*)^2$,
     and hence $\alpha_{i,i}\le 0$.  
    Hence, by Lemma~\ref{lem:nonpositive_certificate}, it follows that $S(\boldsymbol{\alpha},\beta,0) \succeq O$.   
    We next verify that $(\boldsymbol{\alpha}, \beta)$ satisfies \eqref{eq:linear_system_for_feasibility_n}. 
    By substituting $\alpha_{1,1}$, $\alpha_{1,n}$, and $\beta$ into the left-hand side of \eqref{eq:linear_system_for_feasibility_n_1},
    we obtain
    \[
        \alpha_{1,1} z^*_{n+1} + \alpha_{1,n} z^*_{2n} + \beta
        = -(z^*_{n+1})^2 - \Delta_n
          - \sum_{i=1}^n (z_i^*)^2 + \sum_{j=1}^n (z_{n+j}^*)^2.
    \]
    Since
    \[
        \Delta_n = \sum_{t=2}^n \bigl((z_{n+t}^*)^2-(z_t^*)^2\bigr)
    \]
    and $z_1^*=1$, it follows that
    \[
        \alpha_{1,1} z^*_{n+1} + \alpha_{1,n} z^*_{2n} + \beta = -(z_1^*)^2 = -1.
    \]
    For $i \in \{2,\dots,n\}$, we compute
    \[
        \alpha_{i,i-1}z^*_{n+i-1} + \alpha_{i,i}z^*_{n+i}
        = -\frac{\Delta_{i-1}}{z_i^*} - \frac{(z_i^*)^2 - \Delta_{i-1}}{z_i^*}
        = -z^*_i,
    \]
    which implies that \eqref{eq:linear_system_for_feasibility_n_i} is also satisfied.
    Similarly, \eqref{eq:linear_system_for_feasibility_n_npi} holds for $i \in \{2,\dots,n-1\}$ since
    \[
        \alpha_{i,i} z^*_i + \alpha_{i+1,i} z^*_{i+1}
        = -\frac{(z_i^*)^2 - \Delta_{i-1}}{z_{n+i}^*} - \frac{\Delta_i}{z_{n+i}^*}
        = -z^*_{n+i}.
    \]
    From $\alpha_{1,1} z^*_1 + \alpha_{2,1} z^*_{2} = -z^*_{n+1}$,
        \eqref{eq:linear_system_for_feasibility_n_npi} also holds for $i = 1$.
    The last equality \eqref{eq:linear_system_for_feasibility_n_2n} follows from $\alpha_{n,n} z^*_{n} + \alpha_{1,n} z^*_1 = -z^*_{2n}$.
    Therefore, the property~(a) of Proposition~\ref{prop:tightness_on_system} holds.
    In addition, since $\boldsymbol{\alpha} \leq \0$, the property~(b) also holds.
\end{proof}

Theorem~\ref{thm:upperbound_delta_cycle} provides an explicit sufficient condition in terms of
    local ratio bounds along the cycle.
As the admissible range becomes increasingly restrictive for larger $n$,
we next develop  a more flexible sufficient condition.

\subsection{Relaxing the upper bound of \texorpdfstring{$\delta$}{delta} by adding edges}

The local ratio condition in Theorem~\ref{thm:upperbound_delta_cycle} can be written as
    $z^*_{n+j}/z^*_i \in \left[1, 1 + \delta\right]$ for all $(i,j) \in \Omega_\text{cycle}$,
    where $\delta$ is bounded above by $\sqrt{1+\frac{1}{n-1}}-1$.
The upper bound of $z^*_{n+j}/z^*_i$ tends to one as $n$ increases.
Hence, Theorem~\ref{thm:upperbound_delta_cycle} becomes restrictive in large completion problems.
A natural way to enlarge the feasible set of the system is to consider a denser sparsity pattern.  
By adding edges
\[
    \Omega_{\text{add}} \coloneqq \left\{ (1, 2), \ldots, (1, n-1)\right\}
\]
to $\Omega$, the upper-bound $1 + \delta$ in the local ratio condition can be removed.
We therefore consider $\Omega_{\text{cycle}} \cup \Omega_{\text{add}}$, which adds $n-2$ edges. 
This augmented graph clarifies how additional connectivity simplifies
    the construction of nonpositive dual variables.

With $\lambda=0$, the linear system for \eqref{eq:dnn_relaxation} with $\Omega = \Omega_{\text{cycle}} \cup \Omega_{\text{add}}$ becomes   
\begin{subequations} \label{eq:linear_system_for_feasibility_n_dense}
    \begin{empheq}[left = {\empheqlbrace \;}, right = {}]{alignat=4}
        &\sum_{j=1}^n \alpha_{1,j} z^*_{n+j} & &+ \beta  & &= -1, \label{eq:linear_system_for_feasibility_n_dense_1} \\
        &\alpha_{i,i-1} z^*_{n+i-1} & &+ \alpha_{i,i}   z^*_{n+i}                      & &= -z^*_i,     & \quad & i = 2,\ldots,n, \label{eq:linear_system_for_feasibility_n_dense_i} \\
        &\alpha_{1,1}   z^*_{1}     & &+ \alpha_{2,1}   z^*_{2}                        & &= -z^*_{n+1}, \label{eq:linear_system_for_feasibility_n_dense_np1} \\
        &\alpha_{i,i}   z^*_{i}     & &+ \alpha_{i+1,i} z^*_{i+1} + \alpha_{1,i} z^*_1 & &= -z^*_{n+i}, & \quad & i = 2,\ldots,n-1, \label{eq:linear_system_for_feasibility_n_dense_npi} \\
        &\alpha_{n,n}   z^*_{n}     & &+ \alpha_{1,n}   z^*_1                          & &= -z^*_{2n} \label{eq:linear_system_for_feasibility_n_dense_2n}.
    \end{empheq}
\end{subequations}   
    Figure~\ref{fig:edges_corresponding_to_constraints_dense} illustrates the correspondence 
    between the equations in \eqref{eq:linear_system_for_feasibility_n_dense} and the relevant subgraphs. 
    To analyze feasibility, we extend Lemma~\ref{lem:the_others_except_beta} to $\Omega_{\text{cycle}}\cup\Omega_{\text{add}}$.
\begin{figure}
    \centering
    \begin{minipage}[b]{0.23\textwidth}
        \centering
        \begin{tikzpicture}[yscale=-1]
            \draw[opacity=0] (-1.0,-1.0) -- (2.4,2.0) ;
            \draw[thick] (0.0,0.00) -- node[midway,above=-1mm]{\scriptsize $\alpha_{1,1}$} (1.5,0.00) ;
            \draw[thick] (0.0,0.00) -- node[pos=0.75,above=-0.5mm]{\scriptsize $\alpha_{1,2}$} (1.5,0.75) ;
            \draw[thick] (0.0,0.00) -- node[pos=0.3,below=2mm]{\scriptsize $\alpha_{1,n}$} (1.5,1.75) ;
            
            \node at (0.0,-0.5) {$\x$};
            \fill[black] (0,0.00) circle (0.075) node[left]{$1$};
            \node at (1.5,-0.5) {$\y$};
            \fill[black] (1.5,0.00) circle (0.075) node[right]{$1$};
            \fill[black] (1.5,0.75) circle (0.075) node[right]{$2$};
            \node at (1.5,1.15) {$\vdots$};
            \fill[black] (1.5,1.75) circle (0.075) node[right]{$n$};
            
            \draw (2.0,-0.2) -- (2.2,-0.2) -- node[midway,sloped,below] {$n$ edges} (2.2,1.95) -- (2.0,1.95) ;
        \end{tikzpicture}
        \vspace{-3.5ex}
        \subcaption{$i = 1$}
    \end{minipage}
    \begin{minipage}[b]{0.26\textwidth}
        \centering
        \begin{tikzpicture}[yscale=-1]
            \draw[opacity=0] (-1.25,0.0) -- (2.7,1.0) ;
            \draw[thick] (0.0,0.75) -- node[midway,above=0.5mm]{\scriptsize $\alpha_{i,i-1}$} (1.5,0.00) ;
            \draw[thick] (0.0,0.75) -- node[midway,below=-0.25mm]{\scriptsize $\alpha_{i,i}$} (1.5,0.75) ;
            \fill[black] (0,0.75) circle (0.075) node[left]{$i$};
            \fill[black] (1.5,0.00) circle (0.075) node[right]{$i-1$};
            \fill[black] (1.5,0.75) circle (0.075) node[right]{$i$};
        \end{tikzpicture}
        \vspace{-1ex}
        \subcaption{$i \in \{2,\ldots,n\}$}
        \vspace{1ex}
        
        \begin{tikzpicture}[yscale=-1]
            \draw[opacity=0] (-1.25,-0.5) -- (2.7,1.0) ;
            \draw[thick] (0.0,0.00) -- node[midway,above=-1mm]{\scriptsize $\alpha_{1,1}$} (1.5,0.00) ;
            \draw[thick] (0.0,0.75) -- node[midway,below=1mm]{\scriptsize $\alpha_{2,1}$} (1.5,0.00) ;
            \fill[black] (0,0.00) circle (0.075) node[left]{$1$};
            \fill[black] (0,0.75) circle (0.075) node[left]{$2$};
            \fill[black] (1.5,0.00) circle (0.075) node[right]{$1$};
        \end{tikzpicture}
        \vspace{-0.5ex}
        \subcaption{$i = n+1$}
    \end{minipage}
    \begin{minipage}[b]{0.26\textwidth}
        \centering
        \begin{tikzpicture}[yscale=-1]
            \draw[thick] (0.0,0.00) -- node[midway,above=0mm]{\scriptsize $\alpha_{1,i}$} (1.5,1.00) ;
            \draw[thick] (0.0,1.00) -- node[midway,above=-1mm]{\scriptsize $\alpha_{i,i}$} (1.5,1.00) ;
            \draw[thick] (0.0,1.75) -- node[midway,below=1mm]{\scriptsize $\alpha_{i+1,i}$} (1.5,1.00) ;
            
            \node at (0.0,-0.5) {$\x$};
            \fill[black] (0,0.00) circle (0.075) node[left]{$1$};
            \fill[black] (0,1.00) circle (0.075) node[left]{$i$};
            \fill[black] (0,1.75) circle (0.075) node[left]{$i+1$};
            \node at (1.5,-0.5) {$\y$};
            \fill[black] (1.5,1.00) circle (0.075) node[right]{$i$};
        \end{tikzpicture}
        \vspace{-0.8ex}
        \subcaption{$i \in \{n+2,\ldots,2n-1\}$}
    \end{minipage}
    \begin{minipage}[b]{0.20\textwidth}
        \centering
        \begin{tikzpicture}[yscale=-1]
            \draw[thick] (0.0,0.00) -- node[midway,above=1.5mm]{\scriptsize $\alpha_{1,n}$} (1.5,1.75) ;
            \draw[thick] (1.5,1.75) -- node[midway,above=-1mm]{\scriptsize $\alpha_{n,n}$} (0.0,1.75) ;
            
            \node at (0.0,-0.5) {$\x$};
            \fill[black] (0,0.00) circle (0.075) node[left]{$1$};
            \fill[black] (0,1.75) circle (0.075) node[left]{$n$};
            \node at (1.5,-0.5) {$\y$};
            \fill[black] (1.5,1.75) circle (0.075) node[right]{$n$};
        \end{tikzpicture}
        \vspace{-0.6ex}
        \subcaption{$i = 2n$}
    \end{minipage}
    \caption{Graph representations of \eqref{eq:linear_system_for_feasibility_n_dense} associated with the right-hand sides $-z_i^*$} 
    \label{fig:edges_corresponding_to_constraints_dense}
\end{figure}
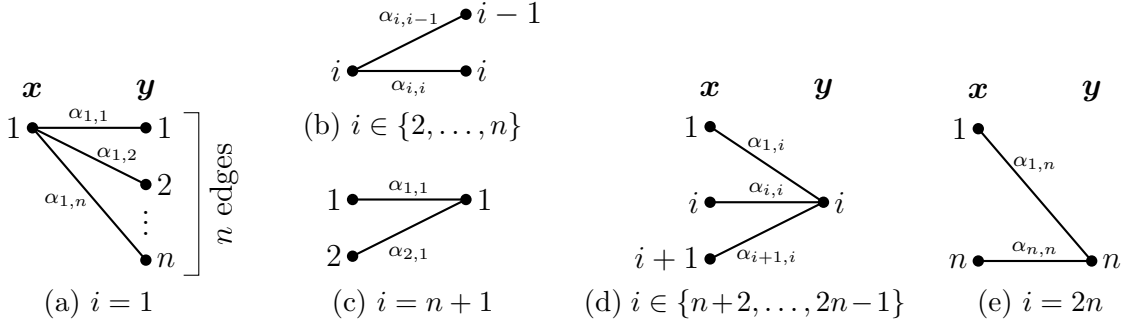

\begin{lemma} \label{lem:the_others_except_beta_dense}   
    Let $(\boldsymbol{\alpha},\beta)$ be a feasible point of \eqref{eq:linear_system_for_feasibility_n_dense}.
    Then, the variables $\alpha_{1,1},\ldots,\alpha_{1,n}$ satisfy
    \begin{equation} \label{eq:difference_of_squared_sum_dense}
        - 1 + \sum_{i=1}^n \left(z^*_i\right)^2 - \sum_{i=1}^n \left(z^*_{n+i}\right)^2 = \sum_{j=1}^n \alpha_{1,j} z^*_{n+j}.
    \end{equation}  
\end{lemma}
\begin{proof}
    Replacing $i$ by $i-1$ in
    \eqref{eq:linear_system_for_feasibility_n_dense_npi} and multiplying by $z^*_{n+i-1}$ gives 
    \[
        z^*_{n+i-1} \left(\alpha_{i-1,i-1} z^*_{i-1} + \alpha_{i,i-1} z^*_{i} + \alpha_{1,i-1} z^*_1\right) = -\left(z^*_{n+i-1}\right)^2
        \quad \text{for $i = 3,\ldots,n$.}
    \]
    By adding \eqref{eq:linear_system_for_feasibility_n_dense_i} multiplied by $(-z^*_i)$,
        for each $i \in \{3,\ldots,n\}$, we have   
    \begin{equation} \label{eq:sum_of_second_and_forth_dense}
        \alpha_{i-1,i-1} z^*_{i-1} z^*_{n+i-1} - \alpha_{i,i} z^*_i z^*_{n+i} + \alpha_{1,i-1} z^*_1 z^*_{n+i-1}
        = -\left(z^*_{n+i-1}\right)^2 + \left(z^*_i\right)^2.
    \end{equation}
    Then, from \eqref{eq:linear_system_for_feasibility_n_dense_i} and \eqref{eq:linear_system_for_feasibility_n_dense_np1}, it follows that
    \begin{align*}
        \alpha_{1,1} z^*_{1} z^*_{n+1}
        &= -\left(z^*_{n+1}\right)^2 - \alpha_{2,1} z^*_2 z^*_{n+1} \\
        &= -\left(z^*_{n+1}\right)^2 + \left(z^*_2\right)^2 + \alpha_{2,2} z^*_2 z^*_{n+2} \\
        &= -\sum_{i=1}^{n-1} \left(z^*_{n+i}\right)^2 + \sum_{i=2}^{n} \left(z^*_i\right)^2
            + \alpha_{n,n} z^*_n z^*_{2n} - \sum_{i=3}^{n} \alpha_{1,i-1} z^*_1 z^*_{n+i-1},
    \end{align*}
    where the last equality follows by recursively applying \eqref{eq:sum_of_second_and_forth_dense}.
    Using \eqref{eq:linear_system_for_feasibility_n_dense_2n}, we obtain 
    \[
        \alpha_{1,1} z^*_{1} z^*_{n+1} =
            -\sum_{i=1}^{n} \left(z^*_{n+i}\right)^2 + \sum_{i=2}^{n} \left(z^*_i\right)^2
            - \sum_{i=2}^{n} \alpha_{1,i} z^*_1 z^*_{n+i}.
    \]  
    Using $z_1^*=1$, we obtain the desired identity.  
\end{proof}

Using Lemma~\ref{lem:the_others_except_beta_dense}, we obtain the following 
    relaxed sufficient condition for tightness of the SDP and DNN relaxations with 
    $\Omega = \Omega_{\text{cycle}}\cup\Omega_{\text{add}}$.
\begin{theorem} \label{thm:no_upperbound_delta_dense}
    Let $\z^*$ be an optimal solution of \eqref{eq:precise_qcqpform_dnn_completion} for an $n \times n$ DNN completion problem, 
        $\KC$ be either $\SymMat_+^{2n}$ or $\DNN^{2n}$,
        and $\Omega$ be the index set of its specified elements.
    Suppose that $\Omega = \Omega_{\text{cycle}} \cup \Omega_{\text{add}}$.
    If $z^*_{n+j} / z^*_i \geq 1$ for all $(i, j) \in \Omega$,
        then the relaxation~\eqref{eq:dnn_relaxation} admits $\z^*\trans{\left(\z^*\right)}$
        as a rank-one optimal solution for each $\theta \in \{0, 1\}$.
    Consequently, the relaxations of \eqref{eq:precise_qcqpform_dnn_completion} and \eqref{eq:qcqpform_dnn_completion} are both tight.
\end{theorem}
\begin{proof}
    Define $\boldsymbol{\alpha}$ by  
    \begin{align*}
        \alpha_{1,1} &\coloneqq - \frac{z^*_{n+1}}{z^*_1}, \\
        \alpha_{i,i} &\coloneqq - \frac{z^*_i}{z^*_{n+i}}, \quad
        \alpha_{1,i} \coloneqq \frac{\left(z^*_i\right)^2 - \left(z^*_{n+i}\right)^2}{z^*_1 z^*_{n+i}}, \quad 
        \alpha_{i,i-1} \coloneqq 0, \quad i=2,\dots,n,
    \end{align*}
    and set
    \[
        \beta \coloneqq - \sum_{i = 1}^n \left(z^*_i\right)^2 + \sum_{j = 1}^n \left(z^*_{n+j}\right)^2.
    \]   
    A direct substitution shows that
     \eqref{eq:linear_system_for_feasibility_n_dense_i}--\eqref{eq:linear_system_for_feasibility_n_dense_2n}
    hold,
    which is analogous to the verification in the proof of Theorem~\ref{thm:upperbound_delta_cycle}.
    The equation~\eqref{eq:linear_system_for_feasibility_n_dense_1} also holds since we have
    \[
        \sum_{j=1}^n \alpha_{1,j} z^*_{n+j} + \beta
        = - \frac{\left(z^*_{n+1}\right)^2}{z^*_1} + \sum_{j=2}^n \frac{\left(z^*_j\right)^2 - \left(z^*_{n+j}\right)^2}{z^*_1} + \beta = -\left(z^*_1\right)^2 = -1.
    \]  
    Hence, $(\boldsymbol{\alpha},\beta)$ satisfies the system~\eqref{eq:linear_system_for_feasibility_n_dense}.
    Since all elements of $\z^*$ are positive, the denominators appearing in the definition of $\boldsymbol{\alpha}$ are positive.
    By the assumption $z_{n+j}^*/z_i^* \geq 1$ for all $(i,j) \in \Omega$, we can easily check that $\boldsymbol{\alpha} \leq \0$.
    Hence, by Lemma~\ref{lem:nonpositive_certificate}, it follows that $S(\boldsymbol{\alpha},\beta,0) \succeq O$.
    The desired result follows from Proposition~\ref{prop:tightness_on_system}.
\end{proof}

The main advantage of Theorem~\ref{thm:no_upperbound_delta_dense} over Theorem~\ref{thm:upperbound_delta_cycle} is 
    that it requires no upper bound on the local ratios. In Theorem~\ref{thm:upperbound_delta_cycle}, the certificate construction 
    relies on a uniform bound of the form $z^*_{n+j}/z^*_i \leq 1+\delta$, which becomes increasingly restrictive as $n$ grows. 
    By contrast, after adding the edges in $\Omega_{\mathrm{add}}$, Theorem~\ref{thm:no_upperbound_delta_dense} requires 
    only the one-sided condition $z^*_{n+j}/z^*_i \geq 1$ on the relevant edges.

\subsection{Cumulative-difference conditions for cycles}

We now develop an alternative tightness condition for cycles $G = (\{1,\ldots,n\},\{1,\ldots,n\},\Omega)$.   
The main idea is to parameterize the construction by cumulative differences between the two blocks of the vector $\z^*$.
As in the proof of Theorem~\ref{thm:upperbound_delta_cycle}, the cumulative differences are defined by
\[
    \Delta_1 \coloneqq 0, \qquad
    \Delta_k \coloneqq \sum_{t=2}^k \left[ \left(z^*_{n+t}\right)^2 - \left(z^*_t\right)^2\right] \; \text{for $k = 2,\ldots,n$}.
\]
For $k = 2,\ldots,n$, the quantity $\Delta_k$ measures how much the elements on the $\y$-side exceed those on the  $\x$-side up to index $k$.
The following theorem shows that, under a suitable bound on these cumulative differences relative to the next node on the cycle,
    a nonpositive certificate can be constructed directly on $\Omega_{\text{cycle}}$.  
\begin{theorem} \label{thm:cumulative_cycle_certificate}
    Let $\z^*$ be an optimal solution of \eqref{eq:precise_qcqpform_dnn_completion} for an $n \times n$ DNN completion problem, 
        $\KC$ be either $\SymMat_+^{2n}$ or $\DNN^{2n}$,
        and $\Omega$ be the index set of its specified elements.
    Suppose that $\Omega = \Omega_{\text{cycle}}$.
    If the following two inequalities    
    \begin{equation} \label{eq:cumulative_cycle_condition_upper}
        0 \leq \Delta_k \leq \left(z^*_{k+1}\right)^2 \qquad (k=2,\ldots,n-1)
    \end{equation}
    and
    \begin{equation} \label{eq:cumulative_cycle_condition_last}
        0 \leq \Delta_n 
    \end{equation}   
    hold, then the relaxation~\eqref{eq:dnn_relaxation} admits $\z^*\trans{\left(\z^*\right)}$ as a rank-one optimal solution for each $\theta \in \{0, 1\}$.
    Consequently, the relaxations of \eqref{eq:precise_qcqpform_dnn_completion} and \eqref{eq:qcqpform_dnn_completion} are both tight.
\end{theorem}
\begin{proof}
    Define $\boldsymbol{\alpha}$ on $\Omega_\text{cycle}$ by  
    \begin{gather*}
        \alpha_{1,1} \coloneqq -\frac{z_{n+1}^*}{z_1^*} = -z_{n+1}^*, \quad       
            \alpha_{1,n} \coloneqq -\frac{\Delta_n}{z_1^* z_{2n}^*} = -\frac{\Delta_n}{z_{2n}^*}, \\
        \alpha_{i,i-1} \coloneqq -\frac{\Delta_{i-1}}{z_i^* z_{n+i-1}^*}, \quad
            \alpha_{i,i} \coloneqq -\frac{\left(z_i^*\right)^2 - \Delta_{i-1}}{z_i^* z_{n+i}^*}
            \quad (i=2,\ldots,n),
    \end{gather*}
    and set
    \[
        \beta \coloneqq -\sum_{i=1}^n \left(z^*_i\right)^2 + \sum_{i=1}^n \left(z^*_{n+i}\right)^2.
    \]
    Since all elements of $\z^*$ are positive in the present setting, all denominators above are positive.
    Moreover, the assumptions imply $\Delta_n \geq 0$ and $0 \leq \Delta_{i-1} \le (z_i^*)^2$ for $i=2,\ldots,n$,
        and therefore $\boldsymbol{\alpha} \leq \0$.
    By Lemma~\ref{lem:nonpositive_certificate}, we have $S(\boldsymbol{\alpha},\beta,0)\succeq O$.
    It remains to verify that $(\boldsymbol{\alpha},\beta)$ satisfies \eqref{eq:linear_system_for_feasibility_n}.
    For $i=2,\ldots,n$, we have
    \[
        \alpha_{i,i-1} z_{n+i-1}^* + \alpha_{i,i} z_{n+i}^*
        = -\frac{\Delta_{i-1}}{z_i^*} - \frac{\left(z_i^*\right)^2 - \Delta_{i-1}}{z_i^*}
        = -z_i^*.
    \]
    On the other hand, for $j=2,\ldots,n-1$, we have
    \[
        \alpha_{j,j} z_j^* + \alpha_{j+1,j} z_{j+1}^*
        = -\frac{\left(z_j^*\right)^2 - \Delta_{j-1}}{z_{n+j}^*} - \frac{\Delta_j}{z_{n+j}^*}
        = -\frac{\left(z_j^*\right)^2 + \Delta_j - \Delta_{j-1}}{z_{n+j}^*}
        = -z_{n+j}^*.
    \]
    Direct substitution leads to
    \begin{align*}
        \alpha_{1,1} z_1^* + \alpha_{2,1} z_2^* &= -z_{n+1}^*, \\
        \alpha_{n,n} z_n^* + \alpha_{1,n} z_1^*
            &= -\frac{\left(z_n^*\right)^2 - \Delta_{n-1}}{z_{2n}^*} - \frac{\Delta_n}{z_{2n}^*} = -z_{2n}^*, \\
        \alpha_{1,1} z_{n+1}^* + \alpha_{1,n} z_{2n}^* + \beta
            &= -\left(z_{n+1}^*\right)^2 - \Delta_n - \sum_{i=1}^n \left(z_i^*\right)^2 + \sum_{i=1}^n \left(z_{n+i}^*\right)^2
            = -\left(z_1^*\right)^2 = -1.
    \end{align*}
    Thus, $(\boldsymbol{\alpha},\beta)$ is a feasible solution of \eqref{eq:linear_system_for_feasibility_n}.
    Therefore, by Proposition~\ref{prop:tightness_on_system}, the desired result follows.
\end{proof}
\noindent
Theorem~\ref{thm:cumulative_cycle_certificate} is qualitatively different from Theorem~\ref{thm:upperbound_delta_cycle}.   
The earlier result is based on a uniform local ratio bound,
    whereas Theorem~\ref{thm:cumulative_cycle_certificate} allows local imbalances as long as their cumulative effect remains controlled.
In this sense, the cumulative condition is more structural: it reflects how the certificate propagates along the cycle.  

Theorem~\ref{thm:cumulative_cycle_certificate}  admits several useful specializations. The following corollary presents 
a simple monotonicity-type condition  expressed directly in terms of the components of $\z^*$.
 We subsequently derive a ratio-based consequence parameterized by a single scalar $\delta$, 
 bringing the condition closer in form to that of Theorem~\ref{thm:upperbound_delta_cycle}.
\begin{coro} \label{coro:local_cycle_certificate_v2}
    Let $\z^*$ be an optimal solution of \eqref{eq:precise_qcqpform_dnn_completion} for an $n \times n$ DNN completion problem,
        $\KC$ be either $\SymMat_+^{2n}$ or $\DNN^{2n}$,
        and $\Omega$ be the index set of its specified elements.
    Suppose that $\Omega = \Omega_{\text{cycle}}$.
    If
    \[
        z^*_t \leq z^*_{n+t} \leq z^*_{t+1} \quad (t=2,\ldots,n-1),
        \quad \text{and}\quad
        z^*_n \leq z^*_{2n},
    \]
    then the assumptions of Theorem~\ref{thm:cumulative_cycle_certificate} are satisfied.
    Consequently, the relaxations of \eqref{eq:precise_qcqpform_dnn_completion} and \eqref{eq:qcqpform_dnn_completion} are both tight.
\end{coro}
\begin{proof}
    For each $t=2,\ldots,n-1$, the assumption gives
    \[
        0 \leq \left(z^*_{n+t}\right)^2 - \left(z^*_t\right)^2
        \leq \left(z^*_{t+1}\right)^2 - \left(z^*_t\right)^2.
    \]
    Hence, for every $k=2,\ldots,n-1$,
    \[
        0 \leq \sum_{t=2}^k \left( \left(z^*_{n+t}\right)^2 - \left(z^*_t\right)^2 \right)
        \leq \sum_{t=2}^k \left( \left(z^*_{t+1}\right)^2 - \left(z^*_t\right)^2 \right)
        = \left(z^*_{k+1}\right)^2 - \left(z^*_2\right)^2
        \leq \left(z^*_{k+1}\right)^2.
    \]
    Moreover, $z^*_n \leq z^*_{2n}$ implies
    \[
        \sum_{t=2}^n \left( \left(z^*_{n+t}\right)^2 - \left(z^*_t\right)^2 \right)
        = \sum_{t=2}^{n-1} \left( \left(z^*_{n+t}\right)^2 - \left(z^*_t\right)^2 \right)
            + \left( \left(z^*_{2n}\right)^2 - \left(z^*_n\right)^2 \right)
        \geq 0.
    \]
    Therefore, all assumptions of Theorem~\ref{thm:cumulative_cycle_certificate} are satisfied.
\end{proof}

\begin{coro} \label{coro:delta_cycle_certificate_v2}
    Let $\z^*$ be an optimal solution of \eqref{eq:precise_qcqpform_dnn_completion} for an $n \times n$ DNN completion problem,
        $\KC$ be either $\SymMat_+^{2n}$ or $\DNN^{2n}$,
        and $\Omega$ be the index set of its specified elements.
    Suppose that $\Omega = \Omega_{\text{cycle}}$.
    If there exists $\delta \geq 0$ such that $(1+\delta)^{2(n-1)} - (1+\delta)^2 \leq 1$ and
    \[
        1 \leq \frac{z^*_{n+j}}{z^*_i} \leq 1 + \delta \qquad \text{for all $(i,j) \in \Omega_{\text{cycle}}$},
    \]
    then the assumptions of Theorem~\ref{thm:cumulative_cycle_certificate} are satisfied.
    Consequently, the relaxations of \eqref{eq:precise_qcqpform_dnn_completion} and \eqref{eq:qcqpform_dnn_completion} are both tight.
\end{coro}
\begin{proof}
    Let $q \coloneqq (1+\delta)^2$.
    By assumption, $q$ satisfies the inequality $q^{n-1} - q \leq 1$.
    For each $t = 2, \ldots, n - 1$, the assumption gives $z^*_t \leq z^*_{n+t} \leq (1 + \delta)z^*_t$,
        which implies 
    \[
        0 \leq \left(z^*_{n+t}\right)^2 - \left(z^*_t\right)^2 \leq \left(q - 1\right)\left(z^*_t\right)^2.
    \]
    Similarly, for each $t = 2, \ldots, n - 1$, the assumption gives $z^*_{n+t} \leq (1 + \delta)z^*_{t+1}$.
    Combining this with $z^*_t \leq z^*_{n+t}$ yields $z^*_t \leq (1 + \delta)z^*_{t+1}$.
    Hence, for every $k \in \{2,\ldots,n-1\}$ and every $t \in \{2,\ldots,k\}$, we have
    \[
        \left(z^*_t\right)^2 \leq q^{k+1-t}\left(z^*_{k+1}\right)^2.
    \]
    Therefore,
    \begin{align*}
        0 \leq \Delta_k \leq (q - 1) \sum_{t=2}^k \left(z^*_t\right)^2
        &\leq (q-1)\sum_{t=2}^k q^{k+1-t} \left(z^*_{k+1}\right)^2 \\
        &= \left(q^k - q\right) \left(z^*_{k+1}\right)^2 & &\leq \left(z^*_{k+1}\right)^2,
    \end{align*}
    where the last inequality follows from $q^{k} - q \leq q^{n-1} - q \leq 1$ for any $k \leq n-1$.
    Finally, since each summand in $\Delta_n$ is nonnegative, we have $\Delta_n \geq 0$.
    Thus, all assumptions of Theorem~\ref{thm:cumulative_cycle_certificate} are satisfied.
\end{proof}

\section{Concluding remarks} \label{sec:conclusion}

We have studied tight conic relaxations for a QCQP formulation of rank-one doubly nonnegative matrix completion.
First, for sparsity patterns whose blocks consist of edges and cycles, we have shown 
that the relaxations associated with the SPN and COP cones are equivalent. 
This result clarifies when these two dual conic formulations yield the same relaxation.
Second, for cycle-type sparsity patterns, we have derived explicit sufficient conditions under which the SDP and DNN relaxations are tight. 
These conditions were obtained by constructing concrete certificates for the complementary-slackness system and 
verifying positive semidefiniteness of the resulting slack matrices.
 In particular, we have presented three types of sufficient conditions: a local ratio condition on the cycle, a condition obtained by adding auxiliary edges, and a cumulative-difference condition that captures the propagation of imbalances along the cycle.
 
Several directions remain for future work. A natural next step is to extend the present analysis beyond 
cycle-type sparsity patterns and identify broader graph classes for which similar certificate constructions 
are possible. It would also be valuable to sharpen the current sufficient conditions and determine whether 
some of them can be strengthened to necessary and sufficient conditions for tightness. Another important question
 is how far the auxiliary-edge construction can be generalized. The results of Section~\ref{sec:tightness_for_conic} 
 suggest that additional connectivity may simplify the structure of feasible certificates, and a systematic understanding 
 of this phenomenon would be of independent interest. Finally, while this paper focuses on the rank-one DNN setting, 
 it would be meaningful to study whether the same conic viewpoint extends to more general sparse completion problems, 
 as well as to approximate recovery under perturbations or noisy observations.

\vspace{0.5cm}

\subsection*{Data availability statement} We did not use any externally collected data in our work. 

\subsection*{Declarations} 

{\bf Conflict of interest} The authors have no conflict of interest to disclose.

\bibliographystyle{abbrvnat} 


\begin{thebibliography}{31}
\providecommand{\natexlab}[1]{#1}
\providecommand{\url}[1]{\texttt{#1}}
\expandafter\ifx\csname urlstyle\endcsname\relax
  \providecommand{\doi}[1]{doi: #1}\else
  \providecommand{\doi}{doi: \begingroup \urlstyle{rm}\Url}\fi

\bibitem[Ai et~al.(2025)Ai, Liang, and Yuan]{ai2025tightness}
W.~Ai, W.~Liang, and J.~Yuan.
\newblock On the tightness of an {SDP} relaxation for homogeneous {QCQP} with
  three real or four complex homogeneous constraints.
\newblock \emph{Mathematical Programming}, 211\penalty0 (1):\penalty0 5--48,
  2025.

\bibitem[Argue et~al.(2023)Argue, K{\dotlessi}l{\dotlessi}n\chige{c}-Karzan,
  and Wang]{Argue2020necessary}
C.~J. Argue, F.~K{\dotlessi}l{\dotlessi}n\chige{c}-Karzan, and A.~L. Wang.
\newblock Necessary and sufficient conditions for rank-one-generated cones.
\newblock \emph{Mathematics of Operations Research}, 48\penalty0 (1):\penalty0
  100--126, 2023.
\newblock \doi{10.1287/moor.2022.1254}.

\bibitem[Azuma et~al.(2022)Azuma, Fukuda, Kim, and Yamashita]{Azuma2021}
G.~Azuma, M.~Fukuda, S.~Kim, and M.~Yamashita.
\newblock Exact {{SDP}} relaxations of quadratically constrained quadratic
  programs with forest structures.
\newblock \emph{Journal of Global Optimization}, 82\penalty0 (2):\penalty0
  243--262, 2022.

\bibitem[Azuma et~al.(2023)Azuma, Fukuda, Kim, and Yamashita]{Azuma2022}
G.~Azuma, M.~Fukuda, S.~Kim, and M.~Yamashita.
\newblock Exact {SDP} relaxations for quadratic programs with bipartite graph
  structures.
\newblock \emph{Journal of Global Optimization}, 86:\penalty0 671 -- 691, 2023.

\bibitem[Azuma et~al.(2025)Azuma, Kim, and Yamashita]{Azuma2025}
G.~Azuma, S.~Kim, and M.~Yamashita.
\newblock Rank-one matrix completion via high-rank matrices in sum-of-squares
  relaxations.
\newblock \emph{Journal of Global Optimization}, 92\penalty0 (2):\penalty0
  321--343, 2025.

\bibitem[Boyd and Vandenberghe(2004)]{Boyd2004}
S.~Boyd and L.~Vandenberghe.
\newblock \emph{Convex Optimization}.
\newblock Cambridge University Press, NY, 2004.
\newblock ISBN 978-0521833783.

\bibitem[Burer and Ye(2020)]{Burer2019}
S.~Burer and Y.~Ye.
\newblock Exact semidefinite formulations for a class of (random and
  non-random) nonconvex quadratic programs.
\newblock \emph{Mathematical Programming}, 181\penalty0 (1):\penalty0 1--17,
  2020.
\newblock \doi{10.1007/s10107-019-01367-2}.

\bibitem[Cand\`{e}s and Plan(2010)]{Candes2010survey}
E.~J. Cand\`{e}s and Y.~Plan.
\newblock Matrix completion with noise.
\newblock \emph{Proceedings of the IEEE}, 98\penalty0 (6):\penalty0 925--936,
  2010.
\newblock \doi{10.1109/JPROC.2009.2035722}.

\bibitem[Cand\`{e}s and Recht(2009)]{Candes2009}
E.~J. Cand\`{e}s and B.~Recht.
\newblock Exact matrix completion via convex optimization.
\newblock \emph{Foundations of {C}omputational {M}athematics}, 9\penalty0
  (6):\penalty0 717--772, 2009.

\bibitem[Cand\`{e}s and Tao(2010)]{Candes2010}
E.~J. Cand\`{e}s and T.~Tao.
\newblock The power of convex relaxation: Near-optimal matrix completion.
\newblock \emph{IEEE Transactions on Information Theory}, 56\penalty0
  (5):\penalty0 2053--2080, 2010.

\bibitem[Cosse and Demanet(2021)]{Cosse2021}
A.~Cosse and L.~Demanet.
\newblock Stable rank-one matrix completion is solved by the level 2 {Lasserre}
  relaxation.
\newblock \emph{Foundations of Computational Mathematics}, 21:\penalty0
  891--940, 2021.

\bibitem[Dickinson and D\"{u}r(2012)]{dickinson2012linear}
P.~J.~C. Dickinson and M.~D\"{u}r.
\newblock Linear-time complete positivity detection and decomposition of sparse
  matrices.
\newblock \emph{SIAM Journal on Matrix Analysis and Applications}, 33\penalty0
  (3):\penalty0 701--720, 2012.
\newblock \doi{10.1137/110848177}.

\bibitem[Dickinson and Gijben(2014)]{dickinson2014computational}
P.~J.~C. Dickinson and L.~Gijben.
\newblock On the computational complexity of membership problems for the
  completely positive cone and its dual.
\newblock \emph{Computational Optimization and Applications}, 57:\penalty0
  403--415, 2014.

\bibitem[Fukuda et~al.(2001)Fukuda, Kojima, Murota, and
  Nakata]{fukuda2001exploiting}
M.~Fukuda, M.~Kojima, K.~Murota, and K.~Nakata.
\newblock Exploiting sparsity in semidefinite programming via matrix completion
  {I}: General framework.
\newblock \emph{SIAM {J}ournal on {O}ptimization}, 11\penalty0 (3):\penalty0
  647--674, 2001.

\bibitem[Grone et~al.(1984)Grone, Johnson, S{\'a}, and
  Wolkowicz]{grone1984positive}
R.~Grone, C.~R. Johnson, E.~M. S{\'a}, and H.~Wolkowicz.
\newblock Positive definite completions of partial hermitian matrices.
\newblock \emph{Linear algebra and its applications}, 58:\penalty0 109--124,
  1984.

\bibitem[Keshavan et~al.(2009)Keshavan, Montanari, and Oh]{Keshavan2009matrix}
R.~H. Keshavan, A.~Montanari, and S.~Oh.
\newblock Matrix completion from noisy entries.
\newblock In \emph{Proceedings of the 23rd International Conference on Neural
  Information Processing Systems}, NIPS'09, pages 952--960, Red Hook, NY, USA,
  2009. Curran Associates Inc.

\bibitem[Kim and Kojima(2003)]{kim2003exact}
S.~Kim and M.~Kojima.
\newblock Exact solutions of some nonconvex quadratic optimization problems via
  {SDP} and {SOCP} relaxations.
\newblock \emph{Computational Optimization and Applications}, 26\penalty0
  (2):\penalty0 143--154, 2003.

\bibitem[K{\dotlessi}l{\dotlessi}n\chige{c}-Karzan and
  Wang(2021)]{Kilinckarzan2021exactness}
F.~K{\dotlessi}l{\dotlessi}n\chige{c}-Karzan and A.~L. Wang.
\newblock Exactness in {SDP} relaxations of {QCQPs}: Theory and applications.
\newblock arXiv:2107.06885, 2021.

\bibitem[Kojima et~al.(2025)Kojima, Kim, and Arima]{Kojima2025extending}
M.~Kojima, S.~Kim, and N.~Arima.
\newblock Extending exact convex relaxations of quadratically constrained
  quadratic programs.
\newblock arXiv:2504.03204v2, 2025.

\bibitem[Lasserre(2001)]{Lasserre2001}
J.~B. Lasserre.
\newblock Global optimization with polynomials and the problem of moments.
\newblock \emph{SIAM {J}ournal on {O}ptimization}, 11\penalty0 (3):\penalty0
  796--817, 2001.

\bibitem[Mazumder et~al.(2010)Mazumder, Hastie, and
  Tibshirani]{Mazumder2010spectral}
R.~Mazumder, T.~Hastie, and R.~Tibshirani.
\newblock Spectral regularization algorithms for learning large incomplete
  matrices.
\newblock \emph{Journal of Machine Learning Research}, 11\penalty0
  (80):\penalty0 2287--2322, 2010.

\bibitem[Merris(1994)]{Merris1994laplacian}
R.~Merris.
\newblock Laplacian matrices of graphs: a survey.
\newblock \emph{Linear Algebra and its Applications}, 197--198:\penalty0
  143--176, 1994.
\newblock \doi{10.1016/0024-3795(94)90486-3}.

\bibitem[Murty and Kabadi(1987)]{Murty1987some}
K.~G. Murty and S.~N. Kabadi.
\newblock Some {NP}-complete problems in quadratic and nonlinear programming.
\newblock \emph{Mathematical Programming}, 39\penalty0 (2):\penalty0 117--129,
  1987.

\bibitem[Parrilo(2003)]{Parrilo2003}
P.~A. Parrilo.
\newblock Semidefinite programming relaxations for semialgebraic problems.
\newblock \emph{Mathematical Programming}, 96\penalty0 (2):\penalty0 293--320,
  2003.

\bibitem[Recht(2011)]{Recht2011simpler}
B.~Recht.
\newblock A simpler approach to matrix completion.
\newblock \emph{Journal of Machine Learning Research}, 12\penalty0
  (104):\penalty0 3413--3430, 2011.

\bibitem[Shaked-Monderer(2016)]{shakedmonderer2016spn}
N.~Shaked-Monderer.
\newblock {SPN} graphs: {W}hen copositive = {SPN}.
\newblock \emph{Linear Algebra and its Applications}, 509:\penalty0 82--113,
  2016.

\bibitem[Shaked-Monderer and Berman(2021)]{ShakedMonderer2021copositive}
N.~Shaked-Monderer and A.~Berman.
\newblock \emph{Copositive and Completely Positive Matrices}.
\newblock World Scientific, 2021.
\newblock ISBN 978-981-12-0436-4.
\newblock \doi{10.1142/11386}.

\bibitem[Sojoudi and Lavaei(2014)]{Sojoudi2014exactness}
S.~Sojoudi and J.~Lavaei.
\newblock Exactness of semidefinite relaxations for nonlinear optimization
  problems with underlying graph structure.
\newblock \emph{SIAM Journal on Optimization}, 24\penalty0 (4):\penalty0
  1746--1778, 2014.
\newblock \doi{10.1137/130915261}.

\bibitem[Srebro et~al.(2004)Srebro, Rennie, and Jaakkola]{Srebro2004maximum}
N.~Srebro, J.~D.~M. Rennie, and T.~S. Jaakkola.
\newblock Maximum-margin matrix factorization.
\newblock In \emph{Proceedings of the 18th International Conference on Neural
  Information Processing Systems}, NIPS'04, pages 1329--1336. MIT Press, 2004.

\bibitem[Wang and
  K{\dotlessi}l{\dotlessi}n\chige{c}-Karzan(2021)]{Wang2021tightness}
A.~L. Wang and F.~K{\dotlessi}l{\dotlessi}n\chige{c}-Karzan.
\newblock On the tightness of {SDP} relaxations of {QCQPs}.
\newblock \emph{Mathematical Programming}, 2021.
\newblock \doi{10.1007/s10107-020-01589-9}.

\bibitem[Wright and Ma(2022)]{WrightMa2022}
J.~Wright and Y.~Ma.
\newblock \emph{High-Dimensional Data Analysis with Low-Dimensional Models:
  Principles, Computation, and Applications}.
\newblock Cambridge University Press, 2022.
\newblock ISBN 978-1-108-48973-7.

\end{thebibliography}

\end{document}